\documentclass{fundam}

\publyear{25}
\papernumber{2}
\volume{194}
\issue{1}
\theDOI{10.46298/fi.12478}

\versionForARXIV

\usepackage{amssymb, amsmath, mathrsfs} 
\usepackage{color}
\usepackage{multicol}
\usepackage{hyperref}
\usepackage{float, caption, subcaption}
\usepackage{tikz}
\usepackage{multirow}
\usepackage{makecell}

\usepackage{calligra}
\usepackage{amsmath}
\usepackage{etoolbox}

 \usepackage[ruled]{algorithm2e}

\usepackage{amsfonts}
\usepackage{amscd,color}
\usepackage{amsmath,amsfonts,amssymb,amscd}
\usepackage{indentfirst,graphicx,epsfig}
\input{epsf}
\usepackage{graphicx}
\usepackage{epstopdf}
\usepackage{caption}
\usepackage{xcolor}

\input{epsf}

\newcommand{\bd}{\begin{description}}
\newcommand{\ed}{\end{description}}
\newcommand{\bi}{\begin{itemize}}
\newcommand{\ei}{\end{itemize}}
\newcommand{\be}{\begin{enumerate}}
\newcommand{\ee}{\end{enumerate}}
\newcommand{\beq}{\begin{equation}}
\newcommand{\eeq}{\end{equation}}
\newcommand{\beqs}{\begin{eqnarray*}}
\newcommand{\eeqs}{\end{eqnarray*}}

\definecolor{DarkGreen}{rgb}{0.2, 0.6, 0.3}

\newtheorem{theorem}{Theorem}[section]

\newtheorem{lemma}{Lemma}[section]
\newtheorem{definition}{Definition}
\newtheorem{corollary}[theorem]{Corollary}
\newtheorem{case}{Case}
\newtheorem{subcase}{Subcase}[case]
\newtheorem{claim}{Claim}

\newtheorem{fact}{Fact}
\newtheorem{proposition}{Proposition}[section]

\begin{document}
\title{{Constructing Disjoint Steiner Trees in Sierpi\'{n}ski Graphs}\thanks{Supported by the National Natural Science Foundation of China 
(Nos.~12201375, 12471329 and 12061059) and the  Qinghai Key Laboratory of
Internet of Things Project (2017-ZJ-Y21).} }

\author{Chenxu Yang \\ 
School of Computer, Qinghai Normal University, Xining, Qinghai 810008, China
\\ {\tt cxuyang@aliyun.com}
\and
Ping Li\thanks{Corresponding author.} \\
School of Mathematics and Statistics, Shaanxi Normal University, Xi'an, Shaanxi, China
\\ {\tt lp-math@snnu.edu.cn} 
\and
Yaping Mao \\
Academy of Plateau Science and Sustainability, and School of Mathematics and Statistics \\
Qinghai Normal University, Xining, Qinghai 810008, China \\
{\tt mao-yaping-ht@ynu.ac.jp} 
\and 
Eddie Cheng \\
Department of Mathematics, Oakland University, Rochester, MI USA 48309 \\
{\tt echeng@oakland.edu}
\and Ralf Klasing \\
Universit\'{e} de Bordeaux, Bordeaux INP, CNRS, LaBRI, UMR 5800, Talence, France \\
{\tt ralf.klasing@labri.fr}}
\maketitle

\runninghead{C.~Yang et al.}{Constructing disjoint Steiner trees in Sierpi\'{n}ski graphs}

\begin{abstract}
\vspace*{-10ex}
Let $G$ be a graph and $S\subseteq V(G)$ with $|S|\geq 2$. Then the trees $T_1, T_2, \cdots, T_\ell$ in $G$ connecting $S$  are
\emph{internally disjoint Steiner trees} (or $S$-Steiner trees) if
$E(T_i) \cap E(T_j )=\emptyset$ and $V(T_i)\cap V(T_j)=S$ for every
pair of distinct integers $1 \leq i, j \leq \ell$. Similarly,
if we only have the condition $E(T_i) \cap E(T_j )=\emptyset$
but without the condition $V(T_i)\cap V(T_j)=S$, then they are \emph{edge-disjoint
Steiner trees} $S$-Steiner trees.
The \emph{generalized $k$-connectivity}, denoted by $\kappa_k(G)$,
of a graph $G$, is defined as $\kappa_k(G)=\min\{\kappa_G(S)|S
\subseteq V(G) \ \textrm{and} \ |S|=k \}$, where $\kappa_G(S)$ is
the maximum number of internally disjoint $S$-Steiner trees. 
The {\it generalized $k$-edge-connectivity}
$\lambda_k(G)$ of $G$ is defined as
$\lambda_k(G)=\min\{\lambda_{G}(S)\,|\,S\subseteq V(G) \ and \ |S|=k\}$, where $\lambda_{G}(S)$ is the maximum number of edge-disjoint Steiner trees
connecting $S$ in $G$.
These concepts are
generalizations of the concepts of connectivity and edge-connectivity, and they can be used as measures
of vulnerability of networks. It is, in general, difficult to compute these generalized connectivities. However, there are precise results for some special classes of graphs.
In this paper, we obtain the exact value of $\lambda_{k}(S(n,\ell))$ for $3\leq k\leq \ell^n$, and the exact value of $\kappa_{k}(S(n,\ell))$ for $3\leq k\leq \ell$, where $S(n, \ell)$ is the Sierpi\'{n}ski graphs with order $\ell^n$. As a direct consequence, these graphs provide additional interesting examples when $\lambda_{k}(S(n,\ell))=\kappa_{k}(S(n,\ell))$.
We also study the some network properties of Sierpi\'{n}ski graphs.
\\[0.2cm]
\end{abstract}

\section{Introduction}
All graphs considered in this paper are undirected, finite and
simple. We refer the readers to \cite{BondyMurty} for graph theoretical
notation and terminology not described here. For a graph $G$, let
$V(G)$ and $E(G)$  denote the set of vertices and the set of edges of $G$, respectively. The
\emph{neighborhood set} of a vertex $v\in V(G)$ is $N_G(v)=\{u\in
V(G)\,|\,uv\in E(G)\}$. The \emph{degree} of a vertex $v$ in $G$ is
denoted by $d(v)=|N_{G}(v)|$. Denote by $\delta(G)$ ($\Delta(G)$)
the minimum degree (maximum degree) of the graph $G$. For a vertex
subset $S\subseteq V(G)$, the subgraph induced by $S$ in $G$ is
denoted by $G[S]$ and similarly $G[V\setminus S]$ for $G\setminus S$
or $G-S$. Especially, $G-v$ is $G[V\setminus \{v\}]$. Let
$\overline{G}$ be the complement of $G$.
For a partition $\mathcal{P}=\{V_1,V_2,\ldots,V_t\}$ of $V(G)$, let $G/\mathcal{P}$ be the graph obtained from $G$ by deleting $\bigcup_{i\in[t]}E(G[V_i])$ and then  identifying each $V_i$, respectively.
For any positive integers $n$, we always
use the convenient notation $[n]$ to denote the set
$\{1,2,\cdots,n\}$.

\subsection{Generalized (edge-)connectivity}

Connectivity and edge-connectivity are two of the most basic
concepts of graph-theoretic measures. Such concepts can be generalized, see, for example, \cite{LMS}.
For a graph $G=(V,E)$ and a set
$S\subseteq V(G)$ of at least two vertices, \emph{an $S$-Steiner
tree} or \emph{a Steiner tree connecting $S$} (or simply, \emph{an
$S$-tree}) is a subgraph $T=(V',E')$ of $G$ that is a tree with
$S\subseteq V'$. Note that when $|S|=2$ a minimal $S$-Steiner tree
is just a path connecting the two vertices of $S$.

Let $G$ be a graph and $S\subseteq V(G)$ with $|S|\geq 2$. Then the trees $T_1, T_2, \cdots, T_\ell$ in $G$ are
\emph{internally disjoint $S$-trees} if
$E(T_i) \cap E(T_j )=\emptyset$ and $V(T_i)\cap V(T_j)=S$ for every
pair of distinct integers $i,j$, $1 \leq i, j \leq \ell$. Similarly,
if we only have the condition $E(T_i) \cap E(T_j )=\emptyset$
but without the condition $V(T_i)\cap V(T_j)=S$, then they are \emph{edge-disjoint
$S$-trees} (Note that while we do not have the condition $V(T_i)\cap V(T_j)=S$, it is still true that $S\subseteq V(T_i)\cap V(T_j)$
as $T_i$ and $T_j$ are $S$-trees.)
The \emph{generalized $k$-connectivity}, denoted by $\kappa_k(G)$,
of a graph $G$, is defined as $\kappa_k(G)=\min\{\kappa_G(S)|S
\subseteq V(G) \ \textrm{and} \ |S|=k \}$, where $\kappa_G(S)$ is
the maximum number of internally disjoint $S$-trees. The \emph{generalized local edge-connectivity}
$\lambda_{G}(S)$ is the maximum number of edge-disjoint $S$-trees
in $G$. The {\it generalized $k$-edge-connectivity}
$\lambda_k(G)$ of $G$ is defined as
$\lambda_k(G)=\min\{\lambda_{G}(S)\,|\,S\subseteq V(G) \ and \ |S|=k\}$.
Since internally disjoint $S$-trees are edge-disjoint but not vice versa, it follows from the definitions that $\kappa_k(G)\leq \lambda_k(G)$.
There are many results on
generalized (edge-)connectivity; see the book \cite{LMbook} by Li and
Mao.

For a graph $G$ and two distinct vertices $x,y$ of $G$, the local connectivity $p_G(x,y)$ of $x$ and $y$ is defined as the maximum number
of pairwise internally disjoint  paths between $x$ and $y$,
and the local edge-connectivity $\lambda_G(x,y)$ is
defined as the maximum number of pairwise edge-disjoint paths between $x$ and
$y$. The connectivity of $G$ is defined as $\kappa(G)=\min\{p_G(x,y)\,|\, x, y\in V(G), \, x\neq y\}$,
and the edge-connectivity of $G$ is defined as
$\lambda(G)=\min\{\lambda_G(x,y)\,|\, x, y\in V(G), \, x\neq y\}$.
It is clear that when $|S|=2$, $\lambda_2(G)$ is just the
standard edge-connectivity $\lambda(G)$ of $G$, $\kappa_2(G)=\kappa(G)$, that is, the standard connectivity of $G$. Thus $\kappa_k(G)$ and $\lambda_k(G)$ are
the generalized connectivity of $G$ and the generalized edge-connectivity of $G$, respectively.

As it is well-known, for any graph $G$, we have polynomial-time
algorithms to find the classical connectivity $\kappa(G)$ and
edge-connectivity $\lambda(G)$. Given two fixed positive integers $k$ and $\ell$, for any
graph $G$ the problem of deciding whether $\lambda_k(G)\geq \ell$
can be solved by a polynomial-time algorithm. If $k \ (k\geq 3)$ is a
fixed integer and $\ell \ (\ell\geq 2)$ is an arbitrary positive
integer, the problem of deciding whether $\kappa(S)\geq \ell$ is
$NP$-complete. For any fixed integer $\ell\geq 3$, given a graph $G$ and a subset
$S$ of $V(G)$, deciding whether there are $\ell$ internally disjoint
Steiner trees connecting $S$, namely deciding whether $\kappa
(S)\geq \ell$, is $NP$-complete. For more details on the computational complexity of generalized
connectivity and generalized edge-connectivity, we refer to the book \cite{LMbook}.

In addition to being a natural combinatorial measure, generalized
$k$-connectivity can be motivated by its interesting interpretation
in practice. For example, suppose that $G$ represents a network. If
one wants to ``connect'' a pair of vertices of $G$ ``minimally'', then a path is
used to ``connect'' them. More generally, if one wants to ``connect'' a set $S$ of
vertices of $G$, with $|S|\geq 3$, ``minimally'', then it is desirable to use a tree to
``connect'' them. Such trees are precisely $S$-trees, which are also used in computer communication
networks (see \cite{Du}) and optical wireless communication networks
(see \cite{Cheng}).

From a theoretical perspective, generalized edge-connectivity is related to Nash-Williams-Tutte
theorem and Menger theorem; see \cite{LMbook}.  The generalized edge-connectivity has applications in $VLSI$ circuit design. In this application, a
Steiner tree is needed to share an electronic signal by a set of
terminal nodes.  Another application, which is our primary focus,
arises in the Internet Domain. Imagine that a given graph $G$
represents a network. We arbitrarily choose $k$ vertices as nodes.
Suppose one of the nodes in $G$ is a {\it broadcaster}, and all
other nodes are either {\it users} or {\it routers} (also called
{\it switches}). The broadcaster wants to broadcast as many streams
of movies as possible, so that the users have the maximum number of
choices. Each stream of movie is broadcasted via a tree connecting
all the users and the broadcaster. So, in essence we need to find
the maximum number Steiner trees connecting all the users and the
broadcaster, namely, we want to get $\lambda (S)$, where $S$ is the
set of the $k$ nodes. Clearly, it is a Steiner tree packing problem.
Furthermore, if we want to know whether for any $k$ nodes the
network $G$ has above properties, then we need to compute
$\lambda_k(G)=\min\{\lambda (S)\}$ in order to prescribe the
reliability and the security of the network. For more details, we refer to the book \cite{LMbook}.

\subsection{Sierpi\'{n}ski graphs}

In 1997, Klav\v{z}ar and Milutinovi\'{c} introduced the
concept of Sierpi\'{n}ski graph $S(n,\ell)$ in \cite{S.Klavzar}.
We denote $n$-tuples $V^{n}$ by the set
$$V^{n}=\{\langle
u_0u_1\cdots u_{n-1}\rangle|\ u_i \in \{0,1,\ldots,\ell-1\} \textrm{ and } i\in \{0,1,\ldots,n-1\}\}.$$
 A {\em word} $u$ of size $n$ are denoted by $\langle
u_0u_1,\cdots,u_{n-1}\rangle$ in which $u_i \in
\{0,\ldots,\ell-1\}$.
The concatenation of two words $u=\langle u_0u_1\cdots u_{n-1}
\rangle$ and $v=\langle v_0v_1\cdots v_{n-1} \rangle$ is denoted by
$uv$.

\begin{definition}
The Sierpi\'{n}ski graph $S(n, \ell)$ is defined as below, for
$n\geq 1$ and $\ell \geq 3$, the vertex set of $S(n, \ell)$ consists
of all $n$-tuples of integers $0,1,\cdots,\ell-1$. That is,
$V(S(n,\ell))=V^n$, where $uv=(u_0u_1\cdots u_{n-1},v_0v_1\cdots v_{n-1} )$ is an edge of $E(S(n,\ell))$
if and only if there exists $d\in \{0,1,\ldots,\ell-1\}$ such that:
$(1)$ $u_j = v_j$, if $j < d$; $(2)$ $u_d \neq v_d$; $(3)$ $u_j =
v_d$ and $v_j = u_d$, if $j >d$.
\end{definition}

\begin{figure}[h]
    \centering
    \includegraphics[width=180pt]{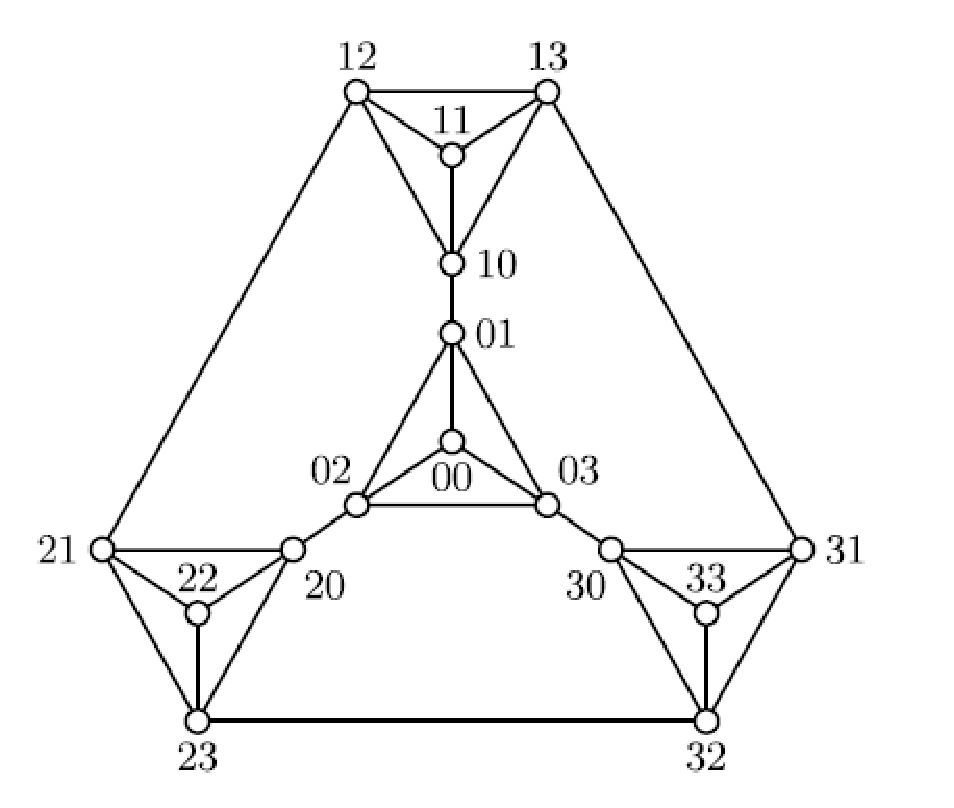}\\
    \caption{$S(2,4)$} \label{S(2,4)}
\end{figure}

Sierpi\'{n}ski graph $S(2,4)$ is shown in Figure \ref{S(2,4)}.
Note that $S(n,\ell)$ can be constructed recursively as follows:
$S(1,\ell)$ is isomorphic to $K_{\ell}$, which vertex set is 1-tuples set $\{0,\ldots,\ell-1\}$. To
construct $S(n,\ell)$ for $n > 1$, we take copies of $\ell$ times $S(n-1,\ell)$
and add the
letter $i$ on the top of the vertices in $i$-th copy of $S(n-1,\ell)$, denoted by $S^{i}(n, \ell)$.
Note that there is exactly one edge (\emph{bridge edge}) between $S^{i}(n, \ell)$ and $S^{j}(n, \ell)$, $i\neq j$, namely the edge between vertices $\langle ij\cdots j \rangle$ and $\langle ji\cdots i\rangle$.

The vertices $(\underbrace{i,i,\cdots,i}_{n})$, $i\in \{0,1,\ldots,\ell-1\}$
are the \emph{extreme vertices} of $S(n, \ell)$.
Note that an extreme vertex $u$ of $S(n,\ell)$ has degree $d(u)=\ell-1$.
For $i\in \{0,\ldots,\ell-1\}$ and $n \geq 2$, let $S^{i}(n-1,\ell)$ denote the
subgraph of $S(n,\ell)$ induced by the vertices of the form
$\{\langle iu_1\cdots u_{n-1}\rangle \,|\,0\leq
u_i \leq \ell-1\}$.
The vertex set $V(S(n,\ell))$ can be partitioned
into $\ell$ parts $V(S^{0}(n-1,\ell)),
V(S^{1}(n-1,\ell)),\ldots,V(S^{\ell-1}(n-1,\ell))$.
For each $0 \leq i \leq \ell-1$, $S^{i}(n-1,\ell)$ is isomorphic to
$S(n-1,\ell)$. Note that $V(S(n,\ell))=V(S^{0}(n-1,\ell))\cup \cdots \cup V(S^{\ell-1}(n-1,\ell))$ and $S(n,\ell)$ is the graph obtained from $S^{0}(n-1,\ell), \ldots ,
S^{\ell-1}(n-1,\ell)$ by adding exactly one edge (\emph{bridge edge}) between $S^{i}(n-1, \ell)$ and $S^{j}(n-1, \ell)$, $i\neq j$, and the bridge edge joins $\langle ij\cdots j \rangle$ and  $ \langle ji\cdots i\rangle$ (notices that $\langle ij\cdots j \rangle$ and  $ \langle ji\cdots i\rangle$ are extreme vertices of $S^{i}(n-1, \ell)$ and $S^{j}(n-1, \ell)$, respectively, if we regard $S^{i}(n-1, \ell)$ and $S^{j}(n-1, \ell)$ as two copies of $S(n-1,\ell)$).

Sierpi\'{n}ski graphs generalize Hanoi graphs which can be viewed as ``discrete'' finite versions of a Sierpi\'nski gasket \cite{Sahu2020, Gu2022}.
Xue considered the Hamiltionicity
and path $t$-coloring of Sierpi\'{n}ski-like graphs in \cite{Xue};
furthermore, they proved that $Val(S(n,k))=Val(S[n,k])=\lfloor
k/2\rfloor$, where $Val(S(n,k))$ is the linear arboricity of
Sierpi\'{n}ski graphs. We remark that although Sierpi\'{n}ski graphs are not regular, they are ``almost'' regular as the extreme vertices have degrees one less than the degrees of non-extreme vertices.

\subsection{Preliminaries and our results}

Chartrand et al. \cite{Chartrand} and Li et al. \cite{LMS} obtained
the exact value of $\kappa_k(K_n)$.

\begin{theorem}[\cite{LMS,Chartrand}]\label{Gary Chartrand etl.}
For every two integers $n$ and $k$ with $2\leq k\leq n$,
$$
\kappa_k(K_n)=\lambda_{k}(K_n)=n-\lceil k/2 \rceil.
$$
\end{theorem}

The following result is on the Hamiltonian decomposition of Sierpi\'{n}ski graphs.
\begin{theorem}{\upshape\cite{Xue}}
\label{Hamitondecomposed1}
$(1)$ For even $\ell \geq 2$, $S(n, \ell)$ can be decomposed into edge disjoint union of $\frac{\ell}{2}$ Hamiltonian paths of which the
end vertices are extreme vertices.

$(2)$ For odd $\ell \geq 3$, there exist $\frac{\ell-1}{2}$ edge-disjoint Hamiltonian paths whose two end vertices are extreme
vertices in $S(n, \ell)$.
\end{theorem}
In fact, Theorem \ref{Hamitondecomposed1}
is used in the proof of Theorem \ref{main},
which give an lower bound for
the generalized $k$-edge connectivity
of Sierpi\'{n}ski graphs $S(n, \ell)$.
We require the following result.

\begin{theorem}[\cite{B.L.chen,XueZUO}]\label{Hamiton-complete}
Suppose that $G$ is a complete graph with $V(G)=\{v_0,\cdots,v_{N-1}\}$. If $N=2n$, then $G$ can be decomposed into $n$ Hamiltonian paths
$$
\{i\in \{1,2,\ldots,n\}:\ L_i=v_{0+i}v_{1+i}v_{2n-1+i}v_{2+i}v_{2n-2+i}\cdots v_{n+1+i}v_{n+1}\},
$$
where the subscripts take modulo $2n$.
If $N=2n+1$, then $G$ can be decomposed into $n$ Hamiltonian paths
$$
\{i\in \{1,2,\ldots,n\}:\ L_i=\ v_{0+i}v_{1+i}v_{2n+i}v_{2+i}v_{2n-1+i}\cdots v_{n+i}v_{n+1+i}\}
$$
and a matching $M=\{v_{n-i}v_{n+1}:\ i\in \{1,2,\ldots,n\}\}$, where the subscripts take modulo $2n+1$.
\end{theorem}
The following result is derived
from Theorem \ref{Hamiton-complete}, and we will use it later.
\begin{corollary}\label{s-Hamil-path}
Let $s$ be an integer with $s\leq \frac{N}{2}$. Suppose that $G$ is the complete graph with $V(G)=\{v_1,\cdots,v_N\}$ and $\mathcal{S}=\{\ \{v_{i_1},v_{i_2}\}:i\in\{1,2,\ldots,s\}\}$ is a collection of pairwise disjoint $2$-subsets of $V(G)$.
Then there are $s$ edge-disjoint Hamiltonian paths $L_1,\cdots,L_s$ such that $v_{i_1},v_{i_2}$ are endpoints of $L_i$.
\end{corollary}

Our main result is  as follows.
\begin{theorem}\label{main}
$(i)$ For $3\leq k\leq \ell$, we have $$\kappa_{k}(S(n,\ell))=
\lambda_{k}(S(n,\ell))=\ell - \left\lceil k/2 \right\rceil.$$

$(ii)$ For $\ell+1 \leq k \leq {\ell}^n$, we have
$$\kappa_{k}(S(n,\ell))\leq \left\lfloor \ell/2 \right\rfloor\mbox{ and }
\lambda_{k}(S(n,\ell))=\left\lfloor \ell/2 \right\rfloor.$$
\end{theorem}

\section{Proof of Theorem \ref{main}}\label{section-main}

This section is organized as follows.
We first introduce some notations, operations and auxiliary graphs that are used in the proof. Then we give a transitional theorem (Theorem \ref{Theorem_lower_bond}), which is the main tool for the proof of Theorem \ref{main}.  Theorem \ref{Theorem_lower_bond} will be proved by constructing an algorithm, and the proof takes up almost entire Section \ref{section-main} (the proof is divided into two parts:  the construction of algorithm and the proof of its correctness). 
Finally, we complete the proof of  Theorem \ref{main} by using Theorem \ref{Theorem_lower_bond}.
 
\subsection{Notations, operations and auxiliary graphs} 
 
For $1\leq s<n$, by the definition of the Sierpi\'{n}ski graph $S(n,\ell)$, we have that $G=S(n,\ell)$ consists of $\ell^{n-s}$ Sierpi\'{n}ski graphs $S(s,\ell)$, and we call each such $S(s,\ell)$ an {\em $s$-atom} of $G$.
Note that $\mathcal{P}=\{V(H):\ H \mbox{ is an } s\mbox{-atom of } G\}$ is a partition of $V(G)$.
We use $G_s$ to denote the graph $G/\mathcal{P}$.
It is obvious that $G_s$ is a simple graph and $G_s=S(n-s,\ell)$.
Note that $G_n$ is a single vertex and $G_0=G$.
For a vertex $u$ of $G_s$, we use $A_{s,u}$ to denote the $s$-atom of $G$ which is contracted into $u$.
See Figures \ref{$G=G_0$},  \ref{$G=G_1$} and  \ref{$G=G_2$} for illustrations.

\begin{table}[!htbp]
\centering
\begin{tabular}{c|p{11cm}}
\hline
  {\bf Notation} & {\bf Corresponding interpretation} \\
 \hline
 $s$-atom &  A copy of $S(s,\ell)$  in $G=S(n,\ell)$, where $0\leq s\leq n$.        \\
 \hline
 $G_s$ & The graph obtained from $G$ by contracting all $\ell^{n-s}$  $s$-atoms in $G$.       \\
 \hline
  $A_{u,s}$  &  The $s$-atom in $G$ which is contracted into $u$, where $u\in V(G_s)$.    \\
 \hline
 $H^u$ & $G_{s-1}$ can be obtained form $G_s$ by expending each $u\in V(G_s)$ to a complete graph $H^u=K_\ell$. \\
 \hline
 $u^e,v^e$ & If $e=uv$ is an edge of $G_s$, then $e$ is also an edge of $G_{s-1}$. let $u_e,v_e$ denote endpoints of $e$ in $G_{s-1}$ such that $u_e\in H^u$ and $v_e\in H^v$. \\
 \hline
 $U$ & The set of labelled vertices in $G$ ($|U|=k$).\\
 \hline
 $U^s$ & The set of labelled vertices in $G_s$.\\
 \hline
 $W^u$ & The set of labelled vertices in $H^u$.\\
 \hline
\end{tabular}
\caption{Notations for and their meanings.}\label{table-1}
\end{table}

Note that $G_{s-1}$ is the graph obtained from $G_s$ by replacing each vertex $u$ with a complete graph $K_\ell$.
We denote by $H^u$ the complete graph replacing $u$ (see Figures \ref{$G=G_1$} and  \ref{$G=G_2$} for illustrations).
For an edge $e=uv$ of $G_s$, $e$ is also an edge of $G_{s-1}$. We use $u_e$ and $v_e$ to denote the ends of $e$ in $G_{s-1}$, respectively, such that $u_e\in V(H^u)$ and $v_e\in V(H^v)$. (See Figure \ref{Constractuv} for an illustration.)

Fix a subset $U$ of $V(G)$ arbitrary such that $|U|=k$, where $3\leq k\leq\ell$. For the graph $G_s$, a vertex $u$ of $G_s$ is {\em labelled} if
$V(A_{s,u})\cap U\neq \emptyset$, and $u$ is {\em unlabelled}
otherwise. We use $U^s$ to denote the set of labelled vertices of
$G_s$. For a labelled vertex $u$ of $G_s$, let $W^u=V(H^u)\cap
U^{s-1}$, that is, the set of labelled vertices in $H^u$.
See Figure \ref{$G=G_0$} as an example, if $U=\{u_{000},u_{001},u_{013},u_{122}\}$, then  $U^1=\{u_{00},u_{01},u_{12}\}$ and $W^{u_0}=\{u_{00},u_{01}\}$ in Figure \ref{$G=G_1$}.
For ease of reading, above notations and their meanings are summarized in Table \ref{table-1}.

We construct an out-branching $\overrightarrow{T}$ with vertex set $V(\overrightarrow{T})=\bigcup_{i\in \{0,1,\ldots,n\}}\{x:\ x\in V(G_i)\}$ and arc set $A(\overrightarrow{T})=\{(x,y):\ y\in V(H^x)\}$.
The root of $\overrightarrow{T}$ is denoted by $v_{root}$ (it is clear that $V(G_n)=\{v_{root}\}$).
For two different vertices $x,y$ of $V(\overrightarrow{T})$, if there is a direct path from $x$ to $y$, then we say that $x\prec y$, and denote the directed path by $x \overrightarrow{T} y$.
The following result is straightforward.
\begin{fact}\label{fact-lian}
If $x_i\in V(\overrightarrow{T})$ for $i\in[p]$ and $x_1\prec x_2\prec\ldots \prec x_p$, then $\sum_{i\in[p]}|W^{x_i}|\leq k+p-1$.
\end{fact}

\begin{figure*}[t!]  
\begin{minipage}[t]{0.95\textwidth}
\center
  \hspace{30pt}
    \begin{tikzpicture}[thick, fill=black, circle=2pt,scale=0.36]
\scriptsize
\coordinate[label=90:$u_{000}  $](v1) at (-6,9) {};
\coordinate[label=90:$u_{001}  $](v2) at (-3,9) {};
\coordinate[label=180:$u_{002}  $](v9) at (-6,6) {};
\coordinate[label=0:$u_{003}  $] (v10) at (-3,6) {};
\coordinate[label=90:$u_{010}  $](v3) at (0,9) {};
\coordinate[label=120:$u_{012}  $](v11) at (0,6) {};
\coordinate[label=90:$u_{011}  $](v4) at (3,9) {};
\coordinate[label=80:$u_{013}  $](v12) at (3,6) {};
\coordinate[label=180:$u_{020}  $](v18) at (-6,3) {};
\coordinate[label=90:$u_{021}  $](v19) at (-3,3) {};
\coordinate[label=90:$u_{030}  $](v20) at (0,3) {};
\coordinate[label=100:$u_{031}  $](v13) at (3,3) {};
\coordinate[label=180:$u_{022}  $](v17) at (-6,0) {};
\coordinate[label=80:$u_{023}  $](v16) at (-3,0) {};
\coordinate[label=-80:$u_{032}  $](v15) at (0,0) {};
\coordinate[label=0:$u_{033}  $](v14) at (3,0) {};
\coordinate[label=90:$u_{100}  $](v5) at (6,9) {};
\coordinate[label=90:$u_{101}  $](v6) at (9,9) {};
\coordinate[label=90:$u_{110}  $](v7) at (12,9) {};
\coordinate[label=90:$u_{111}  $](v8) at (15,9) {};
\coordinate[label=180:$u_{102}  $](v21) at (6,6) {};
\coordinate[label=80:$u_{103}  $](v29) at (9,6) {};
\coordinate[label=180:$u_{112}  $](v30) at (12,6) {};
\coordinate[label=0:$u_{113}  $](v28) at (15,6) {};
\coordinate[label=180:$u_{120}  $](v22) at (6,3) {};
\coordinate[label=90:$u_{121}$](v31) at (9,3) {};
\coordinate[label=90:$u_{130}  $](v32) at (12,3) {};
\coordinate[label=0:$u_{131}$](v27) at (15,3) {};
\coordinate[label=120:$u_{122}  $](v23) at (6,0) {};
\coordinate[label=80:$u_{123}  $] (v24) at (9,0) {} {};
\coordinate[label=-80:$u_{132}  $](v25) at (12,0) {};
\coordinate[label=0:$u_{133}  $](v26) at (15,0) {};
\coordinate[label=180:$u_{200}  $](v33) at (-6,-3) {};
\coordinate[label=90:$u_{201}  $](v51) at (-3,-3) {};
\coordinate[label=90:$u_{210}  $](v52) at (0,-3) {};
\coordinate[label=90:$u_{211}  $](v53) at (3,-3) {};
\coordinate[label=90:$u_{300}  $](v54) at (6,-3) {};
\coordinate[label=90:$u_{301}  $](v57) at (9,-3) {};
\coordinate[label=90:$u_{310}  $](v58) at (12,-3) {};
\coordinate[label=0:$u_{311}  $](v50) at (15,-3) {};
\coordinate[label=180:$u_{202}  $](v34) at (-6,-6) {};
\coordinate[label=80:$u_{203}  $](v63) at (-3,-6) {};
\coordinate[label=180:$u_{212}  $](v64) at (0,-6) {};
\coordinate[label=60:$u_{213}  $](v39) at (3,-6) {};
\coordinate[label=180:$u_{302}  $](v55) at (6,-6) {};
\coordinate[label=60:$u_{303}  $](v59) at (9,-6) {};
\coordinate[label=180:$u_{312}  $](v60) at (12,-6) {};
\coordinate[label=0:$u_{313}  $](v49) at (15,-6) {};
\coordinate[label=180:$u_{220}  $](v35) at (-6,-9) {};
\coordinate[label=90:$u_{221}  $](v36) at (-3,-9) {};
\coordinate[label=90:$u_{230}  $](v37) at (0,-9) {};
\coordinate[label=60:$u_{231}  $](v38) at (3,-9) {};
\coordinate[label=180:$u_{320}  $](v56) at (6,-9) {};
\coordinate[label=90:$u_{321}  $](v61) at (9,-9) {};
\coordinate[label=90:$u_{330}  $](v62) at (12,-9) {};
\coordinate[label=0:$u_{331}  $](v48) at (15,-9) {};
\coordinate[label=180:$u_{222}  $](v40) at (-6,-12) {};
\coordinate[label=-90:$u_{223}  $](v41) at (-3,-12) {};
\coordinate[label=-90:$u_{232}  $](v42) at (0,-12) {};
\coordinate[label=-90:$u_{233}  $](v43) at (3,-12) {};
\coordinate[label=-90:$u_{322}  $](v44) at (6,-12) {};
\coordinate[label=-90:$u_{323}  $] (v45) at (9,-12) {};
\coordinate[label=-90:$u_{ 332}  $](v46) at (12,-12) {};
\coordinate[label=0:$u_{333}  $](v47) at (15,-12) {};
\draw  (v1) edge (v2);
\draw  (v3) edge (v2);
\draw  (v3) edge (v4);
\draw  (v4) edge (v5);
\draw  (v5) edge (v6);
\draw  (v6) edge (v7);
\draw  (v7) edge (v8);
\draw  (v9) edge (v10);
\draw  (v9) edge (v1);
\draw  (v2) edge (v10);
\draw  (v3) edge (v11);
\draw  (v4) edge (v12);
\draw  (v11) edge (v12);
\draw  (v12) edge (v13);
\draw  (v13) edge (v14);
\draw  (v15) edge (v14);
\draw  (v16) edge (v15);
\draw  (v16) edge (v17);
\draw  (v18) edge (v17);
\draw  (v18) edge (v19);
\draw  (v19) edge (v16);
\draw  (v9) edge (v18);
\draw  (v20) edge (v13);
\draw  (v20) edge (v15);
\draw  (v11) edge (v19);
\draw  (v10) edge (v20);
\draw  (v5) edge (v21);
\draw  (v22) edge (v21);
\draw  (v22) edge (v23);
\draw  (v23) edge (v24);
\draw  (v24) edge (v25);
\draw  (v25) edge (v26);
\draw  (v26) edge (v27);
\draw  (v27) edge (v28);
\draw  (v28) edge (v8);
\draw  (v6) edge (v29);
\draw  (v29) edge (v21);
\draw  (v7) edge (v30);
\draw  (v30) edge (v28);
\draw  (v22) edge (v31);
\draw  (v31) edge (v24);
\draw  (v32) edge (v25);
\draw  (v27) edge (v32);
\draw  (v30) edge (v31);
\draw  (v32) edge (v29);
\draw  (v17) edge (v33);
\draw  (v34) edge (v33);
\draw  (v34) edge (v35);
\draw  (v35) edge (v40);
\draw  (v40) edge (v41);
\draw  (v41) edge (v42);
\draw  (v42) edge (v43);
\draw  (v43) edge (v44);
\draw  (v44) edge (v45);
\draw  (v45) edge (v46);
\draw  (v46) edge (v47);
\draw  (v47) edge (v48);
\draw  (v49) edge (v48);
\draw  (v49) edge (v50);
\draw  (v26) edge (v50);
\draw  (v33) edge (v51);
\draw  (v52) edge (v51);
\draw  (v53) edge (v52);
\draw  (v53) edge (v39);
\draw  (v39) edge (v38);
\draw  (v38) edge (v43);
\draw  (v54) edge (v55);
\draw  (v55) edge (v56);
\draw  (v56) edge (v44);
\draw  (v57) edge (v54);
\draw  (v57) edge (v58);
\draw  (v58) edge (v50);
\draw  (v57) edge (v59);
\draw  (v59) edge (v55);
\draw  (v58) edge (v60);
\draw  (v60) edge (v49);
\draw  (v61) edge (v56);
\draw  (v61) edge (v45);
\draw  (v62) edge (v46);
\draw  (v48) edge (v62);
\draw  (v34) edge (v63);
\draw  (v52) edge (v64);
\draw  (v64) edge (v39);
\draw  (v51) edge (v63);
\draw  (v35) edge (v36);
\draw  (v36) edge (v41);
\draw  (v37) edge (v42);
\draw  (v37) edge (v38);
\draw  (v64) edge (v36);
\draw  (v63) edge (v37);
\draw  (v59) edge (v62);
\draw  (v61) edge (v60);
\draw  (v23) edge (v53);
\draw  (v14) edge (v54);
\draw  (v2) edge (v9);
\draw  (v1) edge (v10);
\draw  (v3) edge (v12);
\draw  (v4) edge (v11);
\draw  (v18) edge (v16);
\draw  (v17) edge[loop, looseness=20] (v17);
\draw  (v17) edge (v19);
\draw  (v5) edge (v29);
\draw  (v6) edge (v21);
\draw  (v7) edge (v28);
\draw  (v8) edge (v30);
\draw  (v20) edge (v14);
\draw  (v13) edge (v15);
\draw  (v22) edge (v24);
\draw  (v31) edge (v23);
\draw  (v32) edge (v26);
\draw  (v27) edge (v25);
\draw  (v51) edge (v34);
\draw  (v33) edge (v63);
\draw  (v52) edge (v39);
\draw  (v53) edge (v64);
\draw  (v36) edge (v40);
\draw  (v35) edge (v41);
\draw  (v37) edge (v43);
\draw  (v38) edge (v42);
\draw  (v56) edge (v45);
\draw  (v61) edge (v44);
\draw  (v55) edge (v54);
\draw  (v55) edge[loop, looseness=20] (v55);
\draw  (v54) edge (v59);
\draw  (v57) edge (v55);
\draw  (v50) edge (v60);
\draw  (v58) edge (v49);
\draw  (v48) edge (v46);
\draw  (v47) edge[loop, looseness=20] (v47);
\draw  (v62) edge (v47);
\tikzstyle{bordered} = [draw,outer sep=0,inner sep=1,minimum size=10]
\node [draw,outer sep=0,inner sep=1,minimum size=10] at (-6,9) {};
\node [draw,outer sep=0,inner sep=1,minimum size=10] at (-3,9) {};
\node [draw,outer sep=0,inner sep=1,minimum size=10] at (3,6) {};
\node [draw,outer sep=0,inner sep=1,minimum size=10] at (6,0) {};
\draw[ draw =blue, dashed, line width=0.5pt]  (-2,5) ellipse (6 and 8);
\draw[ draw =blue, dashed, line width=0.5pt]   (11,4) ellipse (7 and 7);
\foreach \p in {v1,v2,v3,v4,v5,v6,v7,v8,v9,v10,v11,v12,v13,v14,v15,v16,v17,v18,v19,v20,
v21,v22,v23,v24,v25,v26,v27,v28,v29,v30,v31,v32,v33,v34,v35,v36,v37,v38,v39,v40,
v41,v42,v43,v44,v45,v46,v47,v48,v49,v50,v51,v52,v53,v54,v55,v56,v57,v58,v59,v60,v61,v62,v63,v64} \fill[opacity=0.75](\p) circle(8pt);

\node at (-1,14) {$ A(2,u_0)$};
\node at (18,10) {$  A(2,u_1)$};
\end{tikzpicture}
\caption{$G=G_0=S(3,4)$}\label{$G=G_0$}
\label{V9-1-vary-S} 
 \end{minipage}\\
\begin{minipage}[t]{0.45\textwidth}
\hspace{10pt}
\begin{tikzpicture}[thick, fill=black, circle=5pt,scale=0.7]
\hspace{5pt}
\scriptsize
\tikzstyle{mynodestyle} = [draw,outer sep=0,inner sep=1,minimum size=10]
 \coordinate[label=90:$u_{00}$ ]  (v1)    at (0,0) {};
 \coordinate[label=90:$u_{01}$](v2) at (2,0) {};
 \coordinate[label=180:$u_{02}$](v4) at (0,-2) {};
 \coordinate[label=60:$u_{03}$](v3) at (2,-2) {};
 \coordinate[label=90:$u_{10}$](v5) at (4,0) {};
 \coordinate[label=90:$u_{11}$](v8) at (6,0) {};
 \coordinate[label=130:$u_{12}$](v6) at (4,-2) {};
 \coordinate[label=0:$u_{13}$](v7) at (6,-2) {};
 \coordinate[label=180:$u_{20}$](v9) at (0,-4) {};
 \coordinate[label=0:$u_{21}$] (v10) at (2,-4) {} {};
 \coordinate[label=180:$u_{30}$](v13) at (4,-4) {};
 \coordinate[label=0:$u_{31}$] (v14) at (6,-4) {} {};
 \coordinate[label=180:$u_{22}$](v12) at (0,-6) {};
 \coordinate[label=-90:$u_{23}$](v11) at (2,-6) {};
 \coordinate[label=-90:$u_{32}$](v16) at (4,-6) {};
 \coordinate[label=0:$u_{33}$](v15) at (6,-6) {};
\draw  (v1) edge (v2);
\draw  (v2) edge (v3);
\draw  (v3) edge (v4);
\draw  (v4) edge (v1);
\draw  (v5) edge (v6);
\draw  (v6) edge (v7);
\draw  (v7) edge (v8);
\draw  (v8) edge (v5);
\draw  (v9) edge (v10);
\draw  (v10) edge (v11);
\draw  (v11) edge (v12);
\draw  (v12) edge (v9);
\draw  (v13) edge (v14);
\draw  (v15) edge (v16);
\draw  (v13) edge (v16);
\draw  (v14) edge (v15);
\draw  (v11) edge (v16);
\draw  (v7) edge (v14);
\draw  (v2) edge (v5);
\draw  (v4) edge (v9);
\draw  (v3) edge (v13);
\draw  (v6) edge (v10);
\draw  (v1) edge (v3);
\draw  (v4) edge (v2);
\draw  (v5) edge (v7);
\draw  (v8) edge (v6);
\draw  (v11) edge (v10);
\draw  (v9) edge (v11);
\draw  (v10) edge (v12);
\draw  (v13) edge (v15);
\draw  (v14) edge (v16);
\foreach \p in {v1,v2,v3,v4,v5,v6,v7,v8,v9,v10,v11,v12,
v13,v14,v15,v16} \fill[opacity=0.75](\p) circle(4pt);
\node [draw,outer sep=0,inner sep=1,minimum size=10] at (0,0) {};
\node [draw,outer sep=0,inner sep=1,minimum size=10] at (2,0) {};
\node [draw,outer sep=0,inner sep=1,minimum size=10] at (4,-2) {};
\draw[draw = red, dashed, line width=1pt]  (1,-1) ellipse (2 and 2);
\draw[ draw =red, dashed, line width=1pt]  (5,-1) ellipse (2 and 2);
\draw  (10,-2) ellipse (0 and 0);

3\node at (2,1.5) {$ H^{u_0}  $};
\node at (7,1) {$  H^{u_1}     $};
\end{tikzpicture}
\caption{The graph $G_1$}\label{$G=G_1$}
\label{V9-2-vary-S} 
\end{minipage}
\hspace{20pt}
\begin{minipage}[t]{0.43\textwidth}
\hspace{20pt}
\begin{tikzpicture}[thick, fill=black, circle=5pt,scale=0.8]
\scriptsize
 \coordinate[label=180:$u_{0}$] (v1) at (-3,3) {};
 \coordinate[label=0:$u_{1}$] (v2) at (2,3) {};
 \coordinate[label=180:$u_{2}$] (v4) at (-3,-2) {};
 \coordinate[label=0:$u_{3}$] (v3) at (2,-2) {};
\draw  (v1) edge (v2);
\draw  (v2) edge (v3);
\draw  (v3) edge (v4);
\draw  (v4) edge (v1);
\draw  (v1) edge (v3);
\draw  (v2) edge (v4);
\foreach \p in {v1,v2,v3,v4} \fill[opacity=0.75](\p) circle(3pt);
\node [draw,outer sep=0,inner sep=1,minimum size=10] at (-3,3) {};
\node [draw,outer sep=0,inner sep=1,minimum size=10] at (2,3) {};
\end{tikzpicture}
\caption{The graph $G_2$}\label{$G=G_2$}
\label{V9-3-vary-S} 
\end{minipage}
\end{figure*}

\subsection{A transitional theorem and its proof}

In order to prove our main result (Theorem \ref{main}), we first prove the following transitional theorem by constructing $\ell-\left\lceil k/ 2 \right\rceil$ internally disjoint $U$-Steiner trees in $G$.

\begin{theorem}\label{Theorem_lower_bond}
Let $G=S(n,\ell)$ and $U\subseteq V(G)$ with $|U|=k$ ($U$ is defined before), and let $c=\ell-\left\lceil k/ 2 \right\rceil$.
For $0\leq s\leq n$, $G_s$ has $c$ internally disjoint $U^s$-Steiner trees $T_1,\cdots,T_c$ such that the following statement holds.
 \begin{enumerate}
   \item [($\star$)]
   For each $u\in U^s$ (say $u\in W^v$ and hence $v$ is a labelled vertex of $G_{s+1}$ if $s\neq n$) and $i\in [c]$, $d_{T_i}(u)\leq 2$.
 \end{enumerate}
\end{theorem}

The rest part of this subsection is the proof of Theorem \ref{Theorem_lower_bond}.

\noindent{\bf Proof of Theorem \ref{Theorem_lower_bond}} 
Suppose that $s'$ is the minimum integer such that $G_{s'}$ has
exactly one labelled vertex (note that $s'$ exists since $G_n$ consists of a single labelled vertex).
This means that $G_{s'-1}$ has
at least two labelled vertices. Let $v$ be the labelled vertex in $G_{s'}$.
Then $U\subseteq A_{s',v}$. Thus, in order to find $c$ internally disjoint $U$-Steiner
trees of $G$ satisfying ($\star$), we only need to find $c$ internally disjoint $U$-Steiner trees in
$A_{s',v}$ satisfying ($\star$). Hence, without loss of generality,
we can assume that
$G$ is a graph
with
\begin{equation}\label{eq-root}
|W^{v_{root}}|=|U^{n-1}|\geq 2.
\end{equation}
Therefore, $|U^i|\geq 2$ for each $i\leq n-1$.

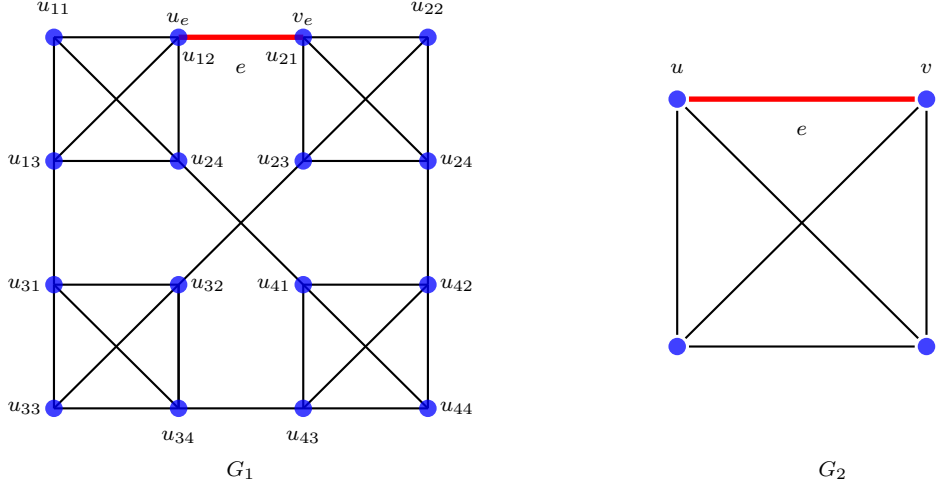
\begin{figure}
  \centering
 \begin{tikzpicture}[thick, fill=blue, circle=5pt,scale=0.82]
\scriptsize\tikzstyle{mynodestyle} = [draw,outer sep=0,inner sep=1,minimum size=10]
\coordinate[label=90:$u_{   11      }$ ]  (v1)    at (0,0) {};
\coordinate[label=-60:$u_{    12     }$](v2) at (2,0) {};
\coordinate[label=180:$u_{   13      }$](v4) at (0,-2) {};
\coordinate[label=0:$u_{     24    }$](v3) at (2,-2) {};
\coordinate[label=-120:$u_{      21   }$](v5) at (4,0) {};
\coordinate[label=90:$u_{       22  }$](v8) at (6,0) {};
\coordinate[label=180:$u_{     23    }$](v6) at (4,-2) {};
\coordinate[label=0:$u_{   24      }$](v7) at (6,-2) {};
\coordinate[label=180:$u_{    31     }$](v9) at (0,-4) {};
\coordinate[label=0:$u_{     32    }$](v10) at (2,-4) {};
\coordinate[label=180:$u_{  41       }$](v13) at (4,-4) {};
\coordinate[label=0:$u_{   42      }$](v14) at (6,-4) {};
\coordinate[label=180:$u_{     33    }$](v12) at (0,-6) {};
\coordinate[label=-90:$u_{   34      }$](v11) at (2,-6) {};
\coordinate[label=-90:$u_{      43   }$](v16) at (4,-6) {};
\coordinate[label=0:$u_{     44    }$](v15) at (6,-6) {};
\node (v17) at (10,-1) {};
\node (v18) at (14,-1) {};
\node (v19) at (14,-5) {};
\node (v20) at (10,-5) {};
\draw [draw=red, line width=2]  (v17) edge (v18);
\draw  (v18) edge (v19);
\draw  (v19) edge (v20);
\draw  (v20) edge (v17);
\draw  (v17) edge (v19);
\draw  (v18) edge (v20);
\draw  (v1) edge (v2);
\draw  (v2) edge (v3);
\draw  (v3) edge (v4);
\draw  (v4) edge (v1);
\draw  (v5) edge (v6);
\draw  (v6) edge (v7);
\draw  (v7) edge (v8);
\draw  (v8) edge (v5);
\draw  (v9) edge (v10);
\draw  (v10) edge (v11);
\draw  (v11) edge (v12);
\draw  (v12) edge (v9);
\draw  (v13) edge (v14);
\draw  (v15) edge (v16);
\draw  (v13) edge (v16);
\draw  (v14) edge (v15);
\draw  (v11) edge (v16);
\draw  (v7) edge (v14);
\draw [draw=red , line width=2] (v2) edge (v5);
\draw  (v4) edge (v9);
\draw  (v3) edge (v13);
\draw  (v6) edge (v10);
\draw  (v1) edge (v3);
\draw  (v4) edge (v2);
\draw  (v5) edge (v7);
\draw  (v8) edge (v6);
\draw  (v11) edge (v10);
\draw  (v9) edge (v11);
\draw  (v10) edge (v12);
\draw  (v13) edge (v15);
\draw  (v14) edge (v16);
\node at (3,-0.5) {$e$};
\foreach \p in {v1,v2,v3,v4,v5,v6,v7,v8,v9,v10,
v11,v12,v13,v14,v15,v16,v17,v18,v19,v20}
\fill[opacity=0.75](\p) circle(4pt);
\node at (2,0.3) {$u_e$};
\node at (4,0.3) {$v_e$};
\node at (10,-0.5) {$u$};
\node at (14,-0.5) {$v$};
\node at (12,-1.5) {$e$};
\node at  (3,-7) {$\huge G_1$};
\node at  (12.5,-7) {$\huge G_2$};
\end{tikzpicture}
\caption{From graph $G_1$ to $G_2$ by $u_ev_e$ contraction edge $uv$.}\label{Constractuv}
\end{figure}

The proof is technical and is via induction. Note that since $G_n$ is a single vertex and $G=G_0$, the induction is from $G_{n}$ to $G_0$.
The basic idea is to use the $U^s$-Steiner trees of $G_s$ to construct appropriate $U^{s-1}$-Steiner trees of $G_{s-1}$. Since each vertex in $G_s$ corresponds to a complete graph in $G_{s-1}$, it
is not a straightforward process to extend $U^s$-Steiner trees of $G_s$ to $U^{s-1}$-Steiner trees of $G_{s-1}$.

If $s=n$, then each $T_i$ in $G_n$ is the empty graph and the result holds.
Thus, suppose $s\leq n-1$.
Hence, the labelled vertex $v$ of $G_{s+1}$ always exists.
The following implies that the result holds for $s=n-1$. Note that $U^{n-1}=W^{v_{root}}$.

\begin{claim}\label{clm-00}
If $s=n-1$, then we can construct $c$ internally disjoint $U^s$-Steiner trees, say $T_1,\ldots,T_c$, such that for each $i\in[c]$ and $u\in U^s$, $d_{T_i}(u)\leq 2$.
Moreover,
\begin{enumerate}
  \item[] $(i)$ If $|U^s|\leq \lceil k/2 \rceil$, then for each $i\in [c]$ and $u\in U^s$, $d_{T_i}(u)=1$;
  \item[] $(ii)$ otherwise, for each $u\in U^s$, there are at most $|U^s|-\lceil k/2 \rceil$ internally disjoint $U^s$-Steiner trees $T_i$ such that $d_{T_i}(u)=2$.
\end{enumerate}
\end{claim}
\begin{proof}
Note that $G_{n-1}=S(1,\ell)=K_{\ell}$.

Suppose
$|U^{n-1}|\leq \left\lceil k/2\right\rceil$.
Then $|V(G_{n-1})-U^{n-1}|\geq \ell-\left\lceil k/2\right\rceil= c$ and we can choose $c$ stars
of $\{x\vee U^{n-1}:\ x\in V(G_{n-1})-U^{n-1}\}$ as internally disjoint
$U^{n-1}$-Steiner trees $T_1,\ldots,T_c$.
It is easy to verify that $d_{T_i}(u)=1$ for $i\in [c]$ and $u\in U^{n-1}$.

Suppose $|U^{n-1}|\geq \lceil k/2 \rceil$.
Since $|U^{n-1}|\leq k$, it follows that $0\leq |U^{n-1}|-\left\lceil k/2\right\rceil\leq \lfloor |U^{n-1}|/2\rfloor$.
By Corollary \ref{s-Hamil-path}, there are $|U^{n-1}|-\left\lceil k/2\right\rceil$
edge-disjoint Hamiltonian paths in $G_{n-1}[U^{n-1}]$.
So, we can choose $c$ internally disjoint $U^{n-1}$-Steiner
trees consisting of $|U^{n-1}|-\left\lceil k/2\right\rceil$
edge-disjoint Hamiltonian paths in $G_{n-1}[U^{n-1}]$ and
$\ell-|U^{n-1}|$ stars $\{x\vee U^{n-1}:\ x\in V(G_{n-1})-U^{n-1}\}$.
It is easy to verify that $d_{T_i}(u)\leq 2$ for $i\in [c]$  and $u\in
U^{n-1}$, and  there are at most
$|U^{n-1}|-\left\lceil k/2\right\rceil$ internally disjoint $U^{n-1}$-Steiner
trees $T_i$ such
that $d_{T_i}(u)=2$.
\end{proof}

Our proof is a recursive process that the $c$ internally disjoint $U^s$-Steiner trees of $G_s$ are constructed by using the $c$ internally disjoint $U^{s+1}$-Steiner trees of $G_{s+1}$. We will find some ways to construct the $c$ internally disjoint $U^s$-Steiner trees such that $(\star)$ holds. In the finial step, the $c$ internally disjoint $U$-Steiner trees in $G_0=G$ will be obtained.
We have proved that the result holds for $s\in\{n,n-1\}$.
Now, suppose that we have constructed $c$ internally disjoint $U^{s+1}$-Steiner trees, say $\mathcal{F}=\{F_1,\ldots,F_{c}\}$, and the trees satisfy ($\star$), where $s\in\{0,\ldots,n-1\}$.
We need to construct $c$ internally disjoint $U^s$-Steiner trees $\mathcal{T}=\{T_1,\ldots,T_c\}$ of $G_s$ satisfying ($\star$).

Recall that each edge of $G_{s+1}$ is also an edge of $G_s$.
In addition, let $E_{u,i}$ denote the set of edges in $F_i$ incident with the labelled vertex $u$ in $G_s$ and let $V_{u,i}=\{u_e:\ e\in E_{u,i}\}$ (recall that the edge $e=uw$ in $G_{s+1}$ is denoted by $e=u_ew_e$ in $G_s$, where $u_e\in V(H^u)$ and $w_e\in V(H^w)$).
Note that $E_{u,i}$ and $V_{u,i}$ are the subsets of $E(G_s)$ and $V(G_s)$, respectively.

Our aim is to construct internally disjoint $U^s$-Steiner trees
$T_1,\ldots,T_{c}$ that is obtained from $U^{s+1}$-Steiner trees
$F_1,\ldots,F_{c}$. So, we need to ``blow'' up each
vertex $w$ of $G_{s+1}$ into $H^w$ and find internally disjoint $W^w$-Steiner trees $F_1^w,\ldots,F_c^w$
of $H^w$ ($F_i^w$ may be the empty graph if $w$ is a unlabelled vertex) such that
$$
\mathcal{T}=\left\{T_i=F_i\cup \bigcup_{w\in V(G_{s+1})}F_i^w:\ i\in [c]\right\}.
$$
Each $F_i^w$ can be constructed as follows.
\begin{itemize}
  \item If $w$ is a labelled vertex of $G_{s+1}$, then $w$ is in each
$U^{s+1}$-Steiner tree of $\mathcal{F}$. We choose $c$ internally disjoint $W^w$-Steiner
trees $F_1^w,\ldots,F_c^w$ of $H^w$ such that $V_{w,i}\subseteq
V(F_i^w)$ for each $i\in [c]$.

\item If $w$ is an unlabelled vertex, then
$w$ is in at most one $U^{s+1}$-Steiner tree of $\mathcal{F}$.
For each $i\in [c]$, if  $w\in V(F_i)$, then let
$F_i^w$ be a spanning tree of $H^w$; if  $w\notin V(F_i)$, then let $F_i^w$ be the empty graph.
\end{itemize}
For
$i\in [c]$, let $E_i=E(F_i)\cup \bigcup_{w\in V(G_{s+1})}E(F_i^w)$ and let
$T_i=G_s[E_i]$. It is obvious that $T_1,\ldots,T_c$ are internally disjoint
$U^s$-Steiner trees. We need to ensure that each labelled vertex
of $G^s$ satisfies ($\star$).

In fact, without loss of generality, we only need to choose an arbitrary labelled vertex $u\in G_{s+1}$ and prove that each vertex of $W^u$ satisfies ($\star$).
This is because $F_i^w$ is clear when $w$ is an unlabelled vertex (in the case $F_i^w$ is either a spanning tree of $H^w$ or the empty graph).
By Eq. (\ref{eq-root}) and Fact \ref{fact-lian}, $|W^u|\leq k-1$ since $u\neq v_{root}$.

Since $d_{F_i}(u)\leq 2$ for each $i\in [c]$ and $u\in U^s$, it follows that $|E_{u,i}|=|V_{u,i}|\leq 2$.
Therefore, we can divide $F_1,\ldots,F_c$ into the following five types on $u$:

{\bf Type 1}: $|V_{u,i}|=1$ and $V_{u,i}\subseteq W^u$.

{\bf Type 2}: $|V_{u,i}|=1$ and $V_{u,i}\nsubseteq W^u$.

{\bf Type 3}: $|V_{u,i}|=2$ and $V_{u,i}\subseteq W^u$.

{\bf Type 4}: $|V_{u,i}|=2$ and $V_{u,i}\cap W^u=\emptyset$.

{\bf Type 5}: $|V_{u,i}|=2$ and $|V_{u,i}\cap W^u|=1$.

If $F_i$ is a graph of Type $j \ (1\leq j\leq 5)$ on $u$, then we also call $F_i^u$ a {\em
$W^u$-Steiner tree of Type $j$}.
Suppose there are $n_j(u)$ trees $F_i$ that belong to Type $j$ on $u$ for $j\in [5]$. Then
\begin{align}\label{n1u-n5u}
\sum_{i=1}^{5}n_j(u)=c.
\end{align}
It is obvious that $n_j(v_{root})=0$ for each $j\in [5]$.
Let
$$
R(u)=V(H^u)-W^u-\bigcup_{i\in[c]}V_{u,i}.
$$
Note that $R(u)$ is the set of unlabelled vertices in $G_s$, which will be used to construct the $W^u$-Steiner trees in $H^u$.
It is clear that
\begin{equation}\label{R(u)}
|R(u)|=\ell-|W^u|-[n_2(u)+2n_4(u)+n_5(u)].
\end{equation}

\begin{figure}[h]
    \centering
    \includegraphics[width=420pt]{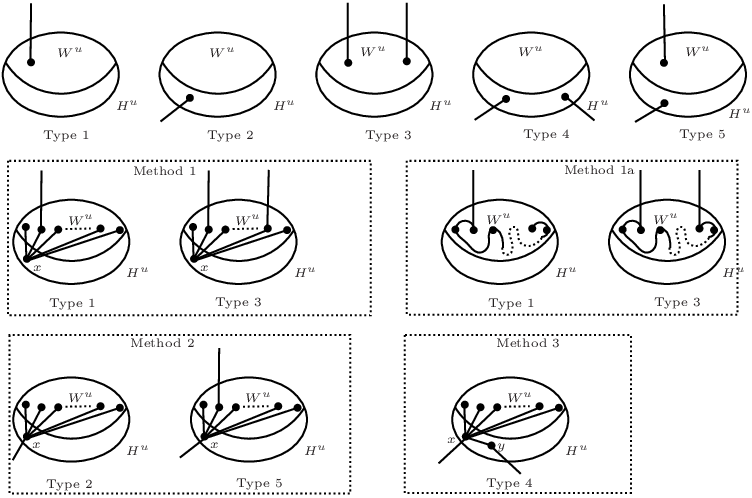}\\
    \caption{Five types and methods 1, 1a, 2 and 3.} \label{Type-Method}
\end{figure}

We have the following five methods to choose $F_1^u,\ldots,F_c^u$ for the labelled vertex $u\in G_{s+1}$ (see Figure \ref{Type-Method}).

{\bf Method 1} If $F_i$ is a graph of Type 1 and $|W^u|\geq 2$, or $F_i$ is a graph of Type 3, then let $F^u_i=x\vee W^u$, where $x\in R(u)$.

{\bf Method 1a} If $F_i$ is a graph of Type 1 and $|W^u|\geq 2$, then let $F^u_i$ be a Hamiltonian path of $H^u[W^u]$ such that the vertex in $V_{u,i}$ is an endpoint of this Hamiltonian path; if $F_i$ is a graph of Type 3, then let $F^u_i$ be a Hamiltonian path of $H^u[W^u]$ such that the endpoints of $F_i$ are two vertices in $V_{u,i}$.

{\bf Method 2} If  $F_i$ is a graph of Type 2 or Type 5, say $x$ is the only vertex of $V_{u,i}$ with $x\notin W^u$, then let $F^u_i=x\vee W^u$.

{\bf Method 3} If $F_i$ is a graph of Type 4, say $V_{u,i}=\{x,y\}$, then let $F^u_i=xy\cup (x\vee W^u)$.

{\bf Method 4} If $F_i$ is a graph of Type 1 and $|W^u|=1$, then let $F_i^u$ be the empty graph.

It is worth noting that the method is deterministic if $F_i^u$ is a graph of Types 2, 4 and 5, or $F_i^u$ is a graph of Type 1 and $|W^u|=1$.
So we firstly construct these $F_i^u$s by using Methods 2, 3 and 4, respectively, and then construct $F^u_i$s of Type 1 with $|W^u|\geq 2$ and construct $F^u_i$s of Type 3 by using Method 1 or Method 1a.

The following is an algorithm constructing $c$ internally disjoint $U$-Steiner trees of $G$.

\smallskip

\begin{algorithm}[H]
\footnotesize
\caption{The construction of $c$ internally disjoint $U$-Steiner trees of $G$}
\label{algorithm:Caculate Stain tree}
\LinesNumbered 
\KwIn{$U$, $c$ and $G=S(n,\ell)$ with $|W^{root}|\geq 2$}
\KwOut{$c$ internally disjoint $U$-Steiner trees $T'_1,T'_2,\ldots, T'_c$ of $G$}
$T'_1,T'_2,\ldots,T'_c$ are empty graphs\;
$s=n$ //(* in initial step, $G^n$ is a single vertex*)\;
\For{$s\geq 1$}{
\For{each vertex $u$ of $G_s$}{
 \If{ $u$ is an unlabelled vertex}{
   \For{ $1\leq i\leq c$}{
      \eIf{ $u$ is a vertex of $F_i$}{
                   $F_i^u$ is a spanning tree of $H^u$\;}{
                  $\ F_i^u$ is is the empty graph;}
		        $T'_i=T'_i\cup F_i^u$\;
                 $i=i+1$\;
                          }
    }

 \If{ $u$ is a labelled vertex}{
  $\mu=|R(u)|$\;
   \For{ $1\leq i\leq c$}{

       \If{ $T'_i$ is of Type 2 or Type 5}{
                    construct $F_i^u$ by using Method 2\;
                 }

       \If{ $T'_i$ is of Type 4}{
       construct $F_i^u$ by using Method 3\;
                                }

       \If{ $T'_i$ is of Type 1 and $|W^u|=1$}{
       construct $F_i^u$ by using Method 4\;
                                              }

        \If{ either $T'_i$ is of Type 1 and $|W^u|\geq 2$, or $T'_i$ is of Type 3}{
           \If{ $\mu>0$}{
           construct $F_i^u$ by using Method 1\;
           $\mu=\mu-1$\; }

            \If{ $\mu\leq 0$}{
             construct $F_i^u$ by using Method 1a\;             }
                                              }

		        $T'_i=T'_i\cup F_i^u$\;
                $i=i+1$\;
                          }
    }

 }
 $s=s-1$\;
}
\end{algorithm}

Algorithm \ref{algorithm:Caculate Stain tree} is an algorithm for finding $c$ internally disjoint $U$-Steiner trees of $G$.
For each step of $s$, the algorithm will construct $c$ internally disjoint $U^{s-1}$-steiner trees $T'_1,T'_2,\ldots,T'_c$.
For convenience, we denote each $T_i'$ by $T_i^s$ after the $s$ step of outer ``for'', that is, $T_1^s,T_2^s,\ldots,T_c^s$ are $c$ internally disjoint $U^s$-Steiner trees in $G^s$ generated in Algorithm \ref{algorithm:Caculate Stain tree}.
If Algorithm \ref{algorithm:Caculate Stain tree} is correct, then Lines 16--40 indicate that $d_{T'_i}(x)\leq 2$ for any $x\in W^u$, and hence ($\star$) always holds.

We now check the correctness of the  algorithm.
Since $F_i^u$ is deterministic when $u$ is an unlabelled vertex (Lines 5--15 of Algorithm \ref{algorithm:Caculate Stain tree}), and $F_i^u$ is deterministic if $u$ is a labelled vertex and either $F_i^u$ is a graph of Types 2, 4 and 5, or $F_i^u$ is a graph of Type 1 and $|W^u|=1$ (Lines 19--27 of Algorithm \ref{algorithm:Caculate Stain tree}), we only need to talk about the labelled vertex $u$ with $|W^u|\geq 2$ and  check the correctness of Lines 28--36 of Algorithm \ref{algorithm:Caculate Stain tree}. That is, to ensure that there exist $n_1(u)$ $F_i^u$s of Type 1 and $n_3(u)$ $F_i^u$s of Type 3.
If $F_i^u$ is constructed by Method 1, then $F_i^u$ is a star with center in $R(u)$ (say the center of the star is $a_i$); if $F_i^u$ is constructed by Method 1a, then $F_i^u$ is a Hamiltonian path of $H^u$ with endpoints in $V_{u,i}$.
Since $a_i$s are pairwise differently and are contained in $R(u)$, and $H^u$ has at most $\lfloor |W^u|/2 \rfloor$ Hamiltonian paths to afford by Corollary \ref{s-Hamil-path}, we only need to ensure that $|R(u)|+\lfloor |W^u|/2 \rfloor\geq n_1(u)+n_3(u)$.
Thus, we only need to prove the following result.

\begin{lemma}\label{clm-main}
For each labelled vertex $u\in V(G_s)$ with $|W^u|\geq 2$, Ineq. \begin{align}\label{equation}
n_1(u)+n_3(u)-|R(u)|\leq \lfloor |W^u|/2\rfloor
\end{align} holds.
\end{lemma}
\begin{proof}
By Eqs. (\ref{n1u-n5u}) and (\ref{R(u)}),
\begin{align*}
n_1(u)+n_3(u)-|R(u)|&=n_1(u)+n_3(u)-[\ell-|W^u|-n_2(u)-2n_4(u)-n_5(u)]\\
&=[n_1(u)+n_2(u)+n_3(u)+n_4(u)+n_5(u)]+n_4(u)+|W^u|-\ell\\
&=c+n_4(u)+|W^u|-\ell.
\end{align*}
Since $c=\ell-\lceil k/2\rceil$, it follows that
\begin{align}\label{equatii}
n_1(u)+n_3(u)-|R(u)|=n_4(u)+|W^u|-\lceil k/2\rceil.
\end{align}

Before the proof of Lemma \ref{clm-main}, we give a series of claims as preliminaries.

\begin{claim}\label{clm-0}
Suppose that $a\in V(v_{root}\overrightarrow{T}u)$ is a labelled vertex with $a\in U^\iota$, where $1\leq \iota\leq n$. If $|W^a|=1$ (say $W^a=\{x\}$), then $n_5(a)\leq 1$. Moreover,
\begin{enumerate}
  \item if $n_5(a)=1$, then $n_4(x)\leq 1$;
  \item if $n_5(a)=0$, then $n_3(x)=n_4(x)=n_5(x)=0$.
\end{enumerate}
\end{claim}
\begin{proof}
Note that $x\in U^{\iota-1}$.
in Algorithm \ref{algorithm:Caculate Stain tree} (Lines 18--39), if $T_i^{\iota}$ is not a tree of Type 5 on $a$, then $F_i^a$ is chosen such that $d_{T_i^{\iota-1}}(x)=1$ (here $W^a=\{x\}$);
if $T^{\iota}_i$ is a tree of Type 5 on $a$, then $F_i^a$ is chosen such that $d_{T_i^{\iota-1}}(x)=2$.
Since $|W^a|=1$, there is at most one $T^{\iota}_i$ of Type 5 on $a$, and hence $n_5(a)\leq 1$.
Furthermore, if $n_5(a)=1$, then there is exact one $T_i^{\iota-1}$ in $G^{\ell-1}$ such that $d_{T_i^{\iota-1}}(x)=2$.
Hence, $n_4(x)\leq 1$.
If $n_5(a)=0$, then $d_{T_i^{\iota-1}}(x)=1$ for each $T_i^{\iota-1}$.
Therefore, $n_3(x)=n_4(x)=n_5(x)=0$.
\end{proof}

\begin{claim}\label{clm-101}
Suppose that $a,b$ are two labelled vertices of $v_{root}\overrightarrow{T}u$ and  $a\in W^b$, where $a\in U^\iota$ for some $\iota\in [n]$. Then there are at most $\max\{1,n_1(b)+n_3(b)-|R(b)|+1\}$ $T_i^{\iota}$s such that $d_{T_i^{\iota}}(a)=2$. Furthermore, $n_4(a)\leq \max\{1,n_1(b)+n_3(b)-|R(b)|+1\}$.
\end{claim}
\begin{proof}
Since $a\in W^b$ and $a\in U^\iota$, it follows that $b\in U^{\iota+1}$.
For an $i\in[c]$ with $d_{T_i^{\iota}}(a)=2$, we have that either $F_i^b$ is a Hamiltonian path of the clique $H^b[W^b]$, or
$a\in V_{b,i}\cap W^b$ for some $i\in [c]$ (recall that $V_{b,i}=\{z\in V(H^b):z\mbox{ is an end vertex of some edge in } E_{b,i}\}$, where $E_{b,i}$ is the set of edges in $T_i^{\iota+1}$ incident with $b$).
Hence, if there are $h$ edge-disjoint $F_i^b$s that are constructed as Hamiltonian paths of $H^b[W^b]$, then there are at most $h+1$ internally disjoint $U^{\iota}$-Steiner trees $T_i^{\iota}$ such that $d_{T_i^{\iota}}(a)=2$.
By the definition of $n_4(a)$, we have that $n_4(a)\leq h+1$.

In Algorithm \ref{algorithm:Caculate Stain tree} (Lines 28--36), if $n_1(b)+n_3(b)-|R(b)|>0$, then there are at most $n_1(b)+n_3(b)-|R(b)|$ $F_i^b$s that are constructed as Hamiltonian paths in $H^b[W^b]$;
if  $n_1(b)+n_3(b)-|R(b)|\leq 0$, there is no $F_i^b$ that is constructed as a Hamiltonian path in $H^b[W^b]$.
Hence, $h\leq \max\{0,n_1(b)+n_3(b)-|R(b)|\}$.
Thus, there are at most $\max\{1,n_1(b)+n_3(b)-|R(b)|+1\}$ internally disjoint $U^{\iota}$-Steiner trees $T_i^{\iota}$ such that $d_{T_i^{\iota}}(a)=2$, and $n_4(a)\leq \max\{1,n_1(b)+n_3(b)-|R(b)|+1\}$.
\end{proof}

\begin{claim}\label{clm-1}
Let $a\in V(v_{root}\overrightarrow{T}u)$ be a labelled vertex, where $a\in U^\iota$ for some $\iota\in [n]$. If $n_4(a)\leq 1$ and $2\leq |W^a|\leq \left\lceil k/2\right\rceil-1$, then $n_1(a)+n_3(a)-|R(a)|\leq 0$. Moreover, for each vertex $x\in W^a$, there is at most one $T_i^{\iota-1}$ such that $d_{T_i^{\iota-1}}(x)=2$.
\end{claim}
\begin{proof}
Since $n_4(a)\leq 1$ and $|W^a|\leq \left\lceil k/2\right\rceil-1$, it follows from Eq. (\ref{equatii}) that
$n_1(a)+n_3(a)-|R(a)|=|W^a|-\left\lceil k/2 \right\rceil+1\leq 0.$
By Claim \ref{clm-101}, for each vertex $x\in W^a$,  there is at most one $T_i^a$ such that $d_{T_i^{\iota-1}}(x)\leq 2$.
\end{proof}

\begin{claim}\label{clm-2}
Let $a\in V(v_{root}\overrightarrow{T}u)$ be a labelled vertex, where $a\in U^\iota$ for some $\iota\in [n]$. If $n_4(a)\leq1$ and $\left\lceil k/2\right\rceil \leq |W^a|\leq k-1 $, then $n_1(a)+n_3(a)-|R(a)|\leq\lfloor |W^a|/2\rfloor$.
Moreover, for each $x\in W^a$,
\begin{enumerate}
  \item if $n_4(a)=1$, then there are at most $|W^a|-\left\lceil k/2\right\rceil+2$ internally disjoint $U^{\iota-1}$-Steiner trees $T_i^{\iota-1}$ such that $d_{T_i^{\iota-1}}(x)=2$;
  \item if $n_4(a)=0$, then there are at most $|W^a|-\left\lceil k/2\right\rceil+1$ internally disjoint $U^{\iota-1}$-Steiner trees $T_i^{\iota-1}$ such that $d_{T_i^{\iota-1}}(x)=2$.
\end{enumerate}
\end{claim}
\begin{proof}
Since $n_4(a)\leq1$, it follows from Eq. (\ref{equatii}) that
\begin{align*}
n_1(a)+n_3(a)-|R(a)|&\leq|W^a|-\left\lceil k/2 \right\rceil+n_4(a)\\
&\leq|W^a|-\left\lceil k/2 \right\rceil+1
\end{align*}
and  the equality indicates $n_4(a)=1$.
If $|W^a|\leq k-2$, then
$$n_1(a)+n_3(a)-|R(a)|\leq |W^a|-\left\lceil (|W^a|+2)/2
\right\rceil+1\leq \left\lfloor |W^a|/2 \right\rfloor;$$
if $|W^a|=k-1$ and $k$ is even, then
$$n_1(a)+n_3(a)-|R(a)|= |W^a|-\left\lceil (|W^a|+1)/2
\right\rceil+1=\left\lfloor (|W^a|+1)/2 \right\rfloor=\left\lfloor |W^a|/2 \right\rfloor;$$
if $|W^a|=k-1$ and $k$ is odd, then
$$n_1(a)+n_3(a)-|R(a)|= (k-1)-\left\lceil k/2
\right\rceil+1\leq \left\lfloor k/2 \right\rfloor=\left\lfloor (k-1)/2 \right\rfloor=\left\lfloor |W^a|/2 \right\rfloor.$$
Therefore, $n_1(a)+n_3(a)-|R(a)|\leq \left\lfloor |W^a|/2\right\rfloor$.
By Eq. (\ref{equatii}) and Claim \ref{clm-101}, if $n_4(a)=1$, then for each vertex $x\in W^a$, there are at most $n_1(a)+n_3(a)-|R(a)|+1=|W^a|-\left\lceil k/2\right\rceil+2$ internally disjoint $U^{\iota-1}$-Steiner trees $T_i^{\iota-1}$ such that $d_{T_i^{\iota-1}}(x)=2$; if $n_4(a)=0$, then for each vertex $x\in W^a$, there are at most $n_1(a)+n_3(a)-|R(a)|+1=|W^u|-\left\lceil k/2\right\rceil+1$ internally disjoint $U^{\iota-1}$-Steiner trees $T_i^{\iota-1}$ such that $d_{T_i^{\iota-1}}(x)=2$.
\end{proof}

\begin{claim}\label{clm-5}
Suppose $a,b\in V(\overrightarrow{T})$, $a\prec b$ and $a\overrightarrow{T}b=az_1z_2\ldots z_pb$, where $a\in U^\iota$ for some $\iota\in [n]$. If  $|W^{z_i}|\leq \lceil k/2 \rceil-1$ for each $i\in[p]$, and either $|W^a|=1$ or $2\leq |W^a|\leq \lceil k/2 \rceil-1$ and $n_4(a)\leq 1$, then $n_4(b)\leq 1$.
\end{claim}
\begin{proof}
Since $a\in U^\iota$, it follows that $z_q\in U^{\iota-q}$ for each $q\in [p]$ and $b\in U^{\iota-p-1}$.
Since $|W^a|=1$ or $2\leq |W^a|\leq \lceil k/2 \rceil-1$ and $n_4(a)\leq 1$, by Claims \ref{clm-0} and \ref{clm-1}, there are at most one $T_i^{\iota-1}$ such that $d_{T_i^{\iota-1}}(z_1)=2$.
Hence, $n_4(z_1)\leq 1$.
Since $|W^{z_1}|\leq \lceil k/2 \rceil-1$ and $n_4(z_1)\leq 1$, by Claims \ref{clm-0} and \ref{clm-1}, there are at most one $T_i^{\iota-2}$ such that $d_{T_i^{\iota-2}}(z_2)=2$.
Hence, $n_4(z_2)\leq 1$.
Repeat this progress, we can get that $n_4(z_p)\leq 1$.
Since $|W^{z_p}|\leq \lceil k/2 \rceil-1$ and $n_4(z_p)\leq 1$, by Claims \ref{clm-0} and \ref{clm-1}, there are at most one  $T_i^{\iota-p-1}$ such that $d_{T_i^{\iota-p-1}}(b)=2$.
Hence, $n_4(b)\leq 1$.
\end{proof}

\begin{claim}\label{clm-6}
Suppose $|W^{v_{root}}|\leq \lceil k/2\rceil$, $b\in V(\overrightarrow{T})$ and $v_{root}\overrightarrow{T}b=v_{root}z_1z_2\ldots z_pb$. If $|W^{z_i}|=1$ for each $i\in [p]$, then $n_4(b)=0$.
\end{claim}
\begin{proof}
Since $|W^{v_{root}}|\leq \lceil k/2\rceil$, by Claim \ref{clm-00}, $d_{T_i^{n-1}}(z_1)=1$ for each $U^{n-1}$-Steiner tree $T_i^{n-1}$.
Hence, $n_3(z_1)=n_4(z_1)=n_5(z_1)=0$.
Since $|W^{z_1}|=1$ and $n_5(z_1)=0$, by the second statement of Claim \ref{clm-0}, $n_3(z_2)=n_4(z_2)=n_5(z_2)=0$.
Since $|W^{z_2}|=1$ and $n_5(z_2)=0$, by the second statement of Claim \ref{clm-0}, $n_3(z_3)=n_4(z_3)=n_5(z_3)=0$.
Repeat this process, we get that $n_4(b)=0$.
\end{proof}

\begin{claim}\label{clm-4}
Suppose that $a,b$ are two labelled vertices and  $a\in W^b$, where $a\in U^\iota$ for some $\iota\in [n]$. If $k$ is even and $|W^b|=|W^a|=k/2$, then $n_1(a)+n_3(a)-|R(a)|\leq\lfloor |W^a|/2\rfloor$ and the following hold.
\begin{enumerate}
  \item If $b=v_{root}$, then for each vertex $x\in W^a$, there is at most one  $T_i^{\iota-1}$ such that $d_{T_i^{\iota-1}}(x)=2$.
  \item If $b\neq v_{root}$, then for each vertex $x\in W^a$, there are at most two  $T_i^{\iota-1}$s such that $d_{T_i^{\iota-1}}(x)=2$.
\end{enumerate}
\end{claim}
\begin{proof}
Suppose $b=v_{root}$. Then by Claim \ref{clm-00}, $n_4(a)=0$.
Thus,
$$n_1(a)+n_3(a)-|R(a)|= n_4(a)+|W^a|-k/2=|W^a|-k/2=0.$$
By Claim \ref{clm-101}, for each vertex $x\in W^a$, there is at most one $T_i^{\iota-1}$ such that $d_{T_i^{\iota-1}}(x)=2$.

Now assume that $b\neq v_{root}$. Without loss of generality, suppose $v_{root}\overrightarrow{T}b=v_{root}v_1v_2\ldots v_pb$. By Fact \ref{fact-lian}, we have that $|W^{root}|=2\leq k/2$ and $|W^{v_i}|=1$ for each $i\in[p]$.
By Claim \ref{clm-6},  $n_4(b)=0$.
By Claim \ref{clm-101},
\begin{align*}
n_4(a)&\leq \max\{1,n_1(b)+n_3(b)-|R(b)|+1\}\\
&\leq \max\{1,n_4(b)+|W^b|-\lceil k/2\rceil+1\}\\
&=1\leq \lfloor|W^a|/2\rfloor.
\end{align*}
By Claim \ref{clm-101} again, for each vertex $x\in W^a$, there are at most two $T_i^{\iota-1}$ such that $d_{T_i^{\iota-1}}(x)=2$.
\end{proof}

\smallskip

With the above preparations, we now prove Lemma \ref{clm-main}. Recall that $u\in V(G_s)$ is a labelled vertex with $|W^u|\geq 2$. Then each vertex of $V(v_{root}\overrightarrow{T}u)$ is also a labelled vertex.
Suppose that $v^*$ is the maximum vertex of $v_{root}\overrightarrow{T}u$ such that one of the following holds (if such vertex $v^*$ exists).
\begin{enumerate}
  \item [($i$)] $|W^{v^*}|=1$,
  \item [($ii$)] $2\leq |W^{v^*}|\leq \lceil k/2 \rceil-1$ and $n_4(v^*)\leq 1$.
\end{enumerate}
We distinguish the following two cases to show this lemma, that is, to prove Ineq. (\ref{equation}) holds (recall that the Ineq. (\ref{equation}) is $n_1(u)+n_3(u)-|R(u)|\leq \lfloor |W^u|/2\rfloor$).
\setcounter{case}{0}
\begin{case}\label{case1}
$v^*$ exists.
\end{case}
\setcounter{subcase}{0}

Let $\overrightarrow{P}=v^*\overrightarrow{T}u=v^* z_1z_2\ldots z_pu$.
By the maximality of $v^*$, we have that $|W^{z_i}|\geq 2$ for each $i\in[p]$.
If $|\overrightarrow{P}|=1$, then $v^*=u$.
Since $|W^u|\geq 2$, it follows from ($ii$) that $2\leq |W^u|\leq \lceil k/2 \rceil-1$ and $n_4(u)\leq 1$. By Claim \ref{clm-1}, we have that $n_1(u)+n_3(u)-|R(u)|\leq 0$, and Ineq. (\ref{equation}) holds.
Thus, we assume that $|\overrightarrow{P}|\geq 2$ below.

By Claim \ref{clm-5}, we have that
\begin{align}\label{eq-n4z1}
n_4(z_1)\leq 1.
\end{align}
Thus, by the maximality of $v^*$, we have that $|W^{z_1}|\geq \lceil k/2 \rceil$.
Since $|W^{v_{root}}|+|W^{z_1}|\leq k+1$ (by Fact \ref{fact-lian}) and $|W^{v_{root}}|\geq 2$, it follows that $|W^{z_1}|\leq k-1$.
Hence,
\begin{align}\label{wz1dx}
\lceil k/2 \rceil\leq |W^{z_1}|\leq k-1.
\end{align}

\setcounter{subcase}{0}
\begin{subcase}\label{sub1-1}
$p=0$.
\end{subcase}

Then $\overrightarrow{P}=v^*\overrightarrow{T}u=v^*u$ and $z_1=u$.
Since $n_4(u)\leq1$ and $\lceil k/2 \rceil\leq |W^u|\leq k-1$ by Ineqs. (\ref{eq-n4z1}) and (\ref{wz1dx}), it follows from Claim \ref{clm-2} that $n_1(u)+n_3(u)-|R(u)|\leq \lfloor |W^u|/2\rfloor$, Ineq. (\ref{equation}) holds.

\setcounter{subcase}{1}
\begin{subcase}\label{sub1-2}
$p=1$.
\end{subcase}

Then $\overrightarrow{P}=v^*\overrightarrow{T}u=v^*z_1u$ and $u=z_2$.
By Fact \ref{fact-lian}, we have that
\begin{align}\label{sub1-2}
|W^{z_1}|+|W^u|+|W^{v_{root}}|-2\leq k.
\end{align}
Hence $|W^{z_1}|+|W^u|\leq k$.
Recall that $|W^{z_1}|\geq \lceil k/2 \rceil$.
If $|W^u|\geq \lceil k/2 \rceil$, then $k$ is even and $|W^{z_1}|=|W^u|=k/2$. By Claim \ref{clm-4}, we have that $n_1(u)+n_3(u)-|R(u)|\leq \lfloor |W^u|/2 \rfloor$.
Therefore, Ineq. (\ref{equation}) holds.
Now, we assume that $2\leq |W^u|\leq \lceil k/2 \rceil-1$.
Since $n_4(z_1)\leq 1$ and $k-1\geq |W^{z_1}|\geq \lceil k/2 \rceil$, by Claim \ref{clm-2},
there are at most $|W^{z_1}|-\left\lceil k/2\right\rceil+2$ internally disjoint $U^s$-Steiner trees $T_i^s$ such that $d_{T_i^s}(u)=2$ (note that $u\in G_s$).
Hence, $n_4(u)\leq |W^{z_1}|-\left\lceil k/2\right\rceil+2$.
Thus, by Eq. (\ref{equatii}),
$$n_1(u)+n_3(u)-|R(u)|=n_4(u)+|W^u|-\left\lceil k/2\right\rceil\leq |W^u|+|W^{z_1}|-2\left\lceil k/2\right\rceil+2.$$
If $|W^u|+|W^{z_1}|\leq k-1$ or $k$ is odd, then $n_1(u)+n_3(u)-|R(u)|\leq 1\leq \lfloor |W^u|/2 \rfloor$,  Ineq. (\ref{equation}) holds.
Thus, assume that $|W^u|+|W^{z_1}|=k$ (recall that $|W^{u}|+|W^{z_1}|\leq k$) and $k$ is even below ($k\geq 4$).
Since $|W^{z_1}|\geq \lceil k/2 \rceil$, it follows that $|W^u|\leq k/2$.
Suppose that $v_{root}\overrightarrow{T}z_p=v_{root}w_1w_2\ldots w_qv^*z_p$.
Since $|W^{v_{root}}|\geq 2$, it follows from Ineq. (\ref{sub1-2}) and Fact \ref{fact-lian} that $|W^{v_{root}}|=2\leq  k/2 $ and $|W^{w_1}|=\ldots=|W^{w_q}|=|W^{v^*}|=1$.
According to Claim \ref{clm-6}, we have that $n_4(u)=0$.
Hence,
\begin{align*}
n_1(u)+n_3(u)-|R(u)|&=n_4(u)+|W^u|-\left\lceil k/2\right\rceil\leq 0\leq \lfloor |W^u|/2\rfloor,
\end{align*}
 the Ineq. (\ref{equation}) holds.

\setcounter{subcase}{2}
\begin{subcase}\label{sub1-3}
$p\geq 2$.
\end{subcase}

Then $u\succeq z_3$.
Recall that  $|W^{z_i}|\geq 2$ for each $i\in[p]$.
Since
\begin{align}\label{equa-1}
|W^{v_{root}}|+|W^{z_1}|+|W^{z_2}|+|W^{z_3}|-3\leq  k
\end{align}
 and $|W^{z_1}|\geq \lceil k/2 \rceil$, it follows  that  $|W^{z_2}|,|W^{z_3}|\leq \lfloor k/2 \rfloor-1$ and $|W^{z_1}|+|W^{z_2}|\leq k-1$.
 On the other hand, since $|W^{z_3}|\geq 2$, it follows that $k\geq 6$.

 Without loss of generality, suppose $z_2\in U^{\iota}$.
 Recall Ineqs.  (\ref{eq-n4z1}) and (\ref{wz1dx}), and combine with Claim \ref{clm-2}, there are at most $|W^{z_1}|-\left\lceil\frac{k}{2}\right\rceil+2$ internally disjoint $U^{\iota}$-Steiner trees $T_i^{\iota}$ such that $d_{T_i^{\iota}}(z_2)=2$.
 Hence, $n_4(z_2)\leq |W^{z_1}|-\left\lceil\frac{k}{2}\right\rceil+2$.
By Claim \ref{clm-101},
 $n_4(z_3)\leq \max\{1,n_4(z_2)+|W^{z_2}|-\lceil k/2 \rceil+1\}$.
However, by the maximality of $v^*$, we have that $n_4(z_3)\geq 2$.
Hence,
\begin{align}\label{xjdbds}
n_4(z_3)\leq n_4(z_2)+|W^{z_2}|-\lceil k/2 \rceil+1\leq |W^{z_1}|+|W^{z_2}|-2\lceil k/2 \rceil+3.
\end{align}
 Since $|W^{z_1}|+|W^{z_2}|\leq k-1$, it follows that if $k$ is odd or $|W^{z_1}|+|W^{z_2}|\leq k-2$, then $n_4(z_3)\leq 1$, a contradiction.
 Hence, $k$ is even and $|W^{z_1}|+|W^{z_2}|=k-1$.
 This implies that $|W^{z_3}|=n_4(z_3)=2$ by Ineqs. (\ref{equa-1}) and (\ref{xjdbds}).
 Since $k\geq 6$, it follows  that $n_4(z_3)+|W^{z_3}|-\lceil k/2 \rceil\leq 1\leq\lfloor |W^{z_3}|/2\rfloor$.
 Recall that $u\succeq z_3$.
 If $v_3=u$, then
 $$n_1(u)+n_3(u)-|R(u)|=n_4(u)+|W^u|-\lceil k/2 \rceil\leq\lfloor |W^u|/2\rfloor,$$
 Ineq. (\ref{equation}) holds.
 If $u\succ z_3$, then since
 \begin{align*}
|W^{v_{root}}|+|W^{z_1}|+|W^{z_2}|+|W^{z_3}|+|W^u|-4\leq  k,
\end{align*}
we have that
\begin{align*}
|W^u|\leq k+4-(|W^{v_{root}}|+|W^{z_1}|+|W^{z_2}|+|W^{z_3}|)=1,
\end{align*}
 a contradiction.

\setcounter{case}{1}
\begin{case}\label{case1}
$v^*$ does not exist.
\end{case}
\setcounter{subcase}{1}

Let $\overrightarrow{P}=v_{root}\overrightarrow{T}u=v_{root}z_1\ldots,z_tu$.
Since $v^*$ does not exist, for each vertex $z\in V(\overrightarrow{P})$, either $2\leq |W^z|\leq \lceil k/2 \rceil-1$ and $n_4(z)\geq 2$, or $|W^z|\geq \lceil k/2 \rceil$.
Since $n_4(v_{root})=0$, it follows that $|W^{v_{root}}|\geq \lceil k/2 \rceil$.
By Claim \ref{clm-00}, $n_4(z_1)\leq |W^{v_{root}}|-\lceil k/2 \rceil$.

\setcounter{subcase}{0}
\begin{subcase}\label{sub2-1}
$|W^{z_1}|\geq \lceil k/2 \rceil$.
\end{subcase}

Since $|W^{v_{root}}|,|W^{z_1}|\geq \lceil k/2 \rceil$, we have that $t\leq 1$.
Otherwise,
$$\sum_{z\in V(\overrightarrow{P})}|W^z|-(|\overrightarrow{P}|-1)>k,$$
which contradicts Fact \ref{fact-lian}.
Moreover, if $t=1$, then $k$ is even, $|W^{v_{root}}|=|W^{z_1}|= k/2$ and $|W^u|=2$.

Suppose that  $t=0$.
Then $u=z_1$, and hence $n_4(u)\leq |W^{v_{root}}|-\lceil k/2 \rceil$. Thus
\begin{align*}
n_1(u)+n_3(u)-|R(u)|&=n_4(u)+|W^u|-\lceil k/2 \rceil\\
&\leq |W^{v_{root}}|+|W^u|-2\lceil k/2\rceil\\
&\leq (k+1)-2\lceil k/2\rceil~(by~ Fact~ \ref{fact-lian})\\
&\leq 1\leq\lfloor |W^u|/2 \rfloor,
\end{align*}
Ineq. (\ref{equation}) holds.

Suppose $t=1$. Then $\overrightarrow{P}=v_{root}z_1u$. Hence,  $k$ is even ($k\geq 4$),  $|W^{v_{root}}|=|W^{z_1}|=k/2$ and $|W^u|=2$.
By the first statement of Claim \ref{clm-4} (here, we regard  $v_{root}$ and $z_1$ as the vertices $b$ and $a$ in Claim \ref{clm-4}, respectively, and then $u$ can be regarded as the vertex $x$ in Claim \ref{clm-4}), we have that $n_4(u)\leq 1$.
Hence,
$$n_1(u)+n_3(u)-|R(u)|=n_4(u)+|W^u|-\lceil k/2 \rceil\leq 1\leq \lfloor |W^u|/2 \rfloor,$$
Ineq. (\ref{equation}) holds.

\eject

\setcounter{subcase}{1}
\begin{subcase}\label{sub2-2}
$2\leq |W^{z_1}|\leq \lceil k/2 \rceil-1$.
\end{subcase}

Recall that  $n_4(z_1)=|W^{v_{root}}|-\lceil k/2 \rceil$.
Since $n_4(z_1)\geq 2$, $|W^{v_{root}}|\geq \lceil k/2 \rceil+2$.
If $|W^{v_{root}}|+|W^{z_1}|=k+1$, then $v=v_{root}$, $u=z_1$ and
$$n_1(u)+n_3(u)-|R(u)| = n_4(u)+|W^u|-\lceil k/2 \rceil\leq 1\leq \lfloor |W^{u}|/2 \rfloor.$$
Thus, suppose $|W^{v_{root}}|+|W^{z_1}|\leq k$.
Then
$$
n_1(z_1)+n_3(z_1)-|R(z_1)|=n_4(z_1)+|W^{z_1}|-\lceil k/2 \rceil\leq 0.$$
If $z_1=u$, then $n_1(u)+n_3(u)-|R(u)|\leq 0,$ Ineq. (\ref{equation}) holds.
Now, assume that $z_1\neq u$.
Then $z_2$ exists.
By Claim \ref{clm-101}, $n_4(z_2)\leq 1$.
Since $z_2$ is not a candidate of $v^*$, it follows that $|W^{z_2}|\geq \lceil k/2 \rceil$.
Since $|W^{v_{root}}|\geq \lceil k/2 \rceil+2,|W^{z_2}|\geq \lceil k/2 \rceil$ and $|W^{z_1}|\geq 2$, it follows that $|W^{v_{root}}|+|W^{z_1}|+|W^{z_2}|-2\geq k+2$,
which contradicts Fact \ref{fact-lian}.
\end{proof}
With the conclusion of  Lemma \ref{clm-main}, the proof of Theorem \ref{Theorem_lower_bond} is completed.

\subsection {Proof of Theorem \ref{main}}

We first consider $3\leq k\leq \ell$. The lower bound $\ell-\lceil k/2\rceil\leq\kappa_{k}(S(n,\ell))\leq\lambda_{k}(S(n,\ell)) $ can be obtained from Theorem \ref{Theorem_lower_bond} directly.
For the upper bounds of $\kappa_{k}(S(n,\ell))$ and $\lambda_{k}(S(n,\ell))$, consider the graph $G_{n-1}$.
Let $V(G_{n-1})=\{u_1,\ldots,u_\ell\}$, $U=\{x_1,\ldots, x_k\}$ and $x_i\in A_{n-1,u_i}$, where $k \leq \ell$.
Suppose there are $p$ edge-disjoint $U$-Steiner trees $T'_1,\ldots,T'_p$ of $G=S(n,\ell)$ and $\mathcal{P}=\{A_{n-1,u}: u\in V(G_{n-1})\}$.
Let $T^*_i=T'_i/\mathcal{P}$ for $i\in[p]$.
Then $T^*_1,\ldots,T^*_p$ are edge-disjoint connected graphs of $G_{n-1}$ containing $\{u_1,\ldots,u_k\}$.
Thus, $p\leq \lambda_{k}(G_{n-1})=\lambda_{k}(K_\ell)=\ell-\lceil
k/2\rceil$.
Therefore, $\kappa_{k}(S(n,\ell))\leq\lambda_{k}(S(n,\ell))\leq  \ell-\lceil
k/2\rceil$, the upper bound follows.

Now consider the case $\ell+1 \leq k \leq {\ell}^n$.
By Theorem \ref{Hamitondecomposed1}, $\lambda_{k}(S(n,\ell))\geq\lfloor
\ell/2\rfloor$.
Since $\lambda_{k}(S(n,\ell))\leq \lambda_{\ell}(S(n,\ell))=\lfloor
\ell/2\rfloor$, it follows that $\kappa_{k}(S(n,\ell))\leq \lambda_{k}(S(n,\ell))\leq \lfloor
\ell/2\rfloor$.
Therefore, $\lambda_{k}(S(n,\ell))=\lfloor\ell/2\rfloor$ and $\kappa_{k}(S(n,\ell))\leq\lfloor
\ell/2\rfloor$.
The proof is completed.

\section{Some network properties}

Generalized connectivity is a graph parameter to measure the stability of networks.
In the following part, we will give the following other properties of Sierpi\'{n}ski graphs.

The Sierpi\'{n}ski graph(networks)  is obtained after $t$ iteration as $S(t,\ell)$ that has $N_t$
nodes and $E_t$ edges, where $t=0,1,2,\cdots,T-1$, and $T$ is the
total number of iterations, and our generation process can be
illustrated as follows.

\begin{figure}[!htbp]
  \centering
  \includegraphics[width=8cm]{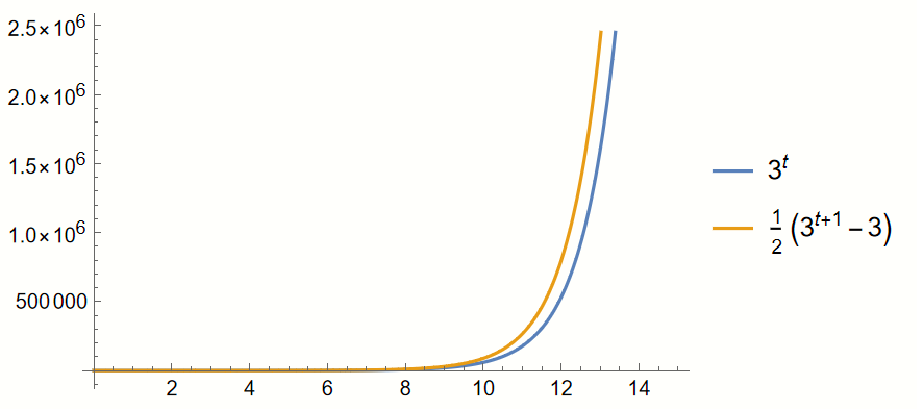}
 \caption{The Function of $E_t=3^t$ and $N_t=\frac{3^{t+1}-3}{2}$}\label{Fig:EV}
\end{figure}

\textbf{Step 1:} Initialization. Set $t=0$, $G_1$ is a complete graph of
order $\ell$, and thus $N_1=\ell$ and $E_1=\binom{\ell}{2}$.
Set
$G_1=S(1,\ell)$.

\textbf{Step 2:} Generation of $G_{t+1}$ from $G_t$.
Let
$S^{1}(t,\ell), S^{2}(t,\ell),\ldots, S^{\ell}(t,\ell)$ be all Sierpi\'{n}ski graphs added at Step $t$, where $S^{i}(t,\ell) \cong S(t,\ell)$($1\leq i\leq \ell$). At Step $t+1$,
we add one edge (\emph{bridge edge}) between $S^{i}(t, \ell)$ and $S^{j}(t, \ell)$, $i\neq j$, namely the edge between vertices $\langle ij\cdots j \rangle$ and $\langle ji\cdots i\rangle$.

For Sierpi\'{n}ski graphs $G_t=S(t,\ell)$,
its order and size are
$N_t=\ell^t$ and
$E_t=\frac{\ell^{t+1}-\ell}{2}$, respectively; see
Table \ref{Tab:degreedistri}
and Figure \ref{Fig:EV}(for $\ell=3$).
\begin{table}[!htbp]
\centering
\caption{The size and order of $G_t$ for $\ell=3$}
\label{Tab:degreedistri}
\begin{tabular}{cccccccc}
\hline
$t$   & 1 & 2  &3 &4  &5   &6 & t \\
\hline
$N_t$ & 3 & 9  &27 &81 &242&729  &$\ell^t$\\
$E_t$ & 3 & 12 &39 &120 &363 & 1092
& $\frac{\ell^{t+1}-\ell}{2}$\\  \hline
\end{tabular}
\end{table}

The degree distribution for $t$ times are
\begin{equation}\label{Equ:DegDistri}
\left(\underbrace{\ell-1, \ell-1, \cdots ,\ell-1}_{\ell}, \underbrace{\ell ,\ell ,\ell ,
\cdots,\ell }_{\ell^t-\ell}\right).
\end{equation}
From Equation \ref{Equ:DegDistri},
the instantaneous degree distribution
is
$P(\ell-1,t)= 1/\ell^{t-1}$ for $t=2,\cdots,T$
and
$P(\ell,t)= (\ell^t-\ell)/\ell^t$ for $t=2,\cdots,T$.
Note that the density of Sierpi\'{n}ski graphs is
$\rho=E_t/\binom{N_t}{2} \rightarrow 0$
for $t \rightarrow +\infty$.
For large enough $\ell$ and any  $1\leq k\leq \ell$,
we have
$|\{v\in V(G)|d_{S(t,\ell)}(v)|\geq k\}\approx|V(S(t,\ell))|$.

\begin{theorem}{\upshape \cite{S.Klavzar}}
\label{Thm:VeC}
If $n \in \mathbb{N}$ and $G$ is a graph, then $\kappa\left(S(n,G) \right)=\kappa(G)$ and $\lambda\left(S(n,G) \right)=
\lambda(G)$.
\end{theorem}
From Theorem \ref{Thm:VeC}, we have $\kappa\left(S(n,\ell) \right)=\kappa(K_{\ell})=\ell-1$ and $\lambda\left(S(n,\ell)\right)=
\lambda(K_{\ell})=\ell-1$.
Note that
$\lambda_{k}(S(n,\ell))=\ell - \left\lceil k/2
\right\rceil$
and
$\kappa_{k}(S(n,\ell))=\ell - \left\lceil k/2 \right\rceil$;
see Figure \ref{Fig:edge-VerConn}.
\begin{figure}[!htbp]
  \centering
  \includegraphics[width=10cm]{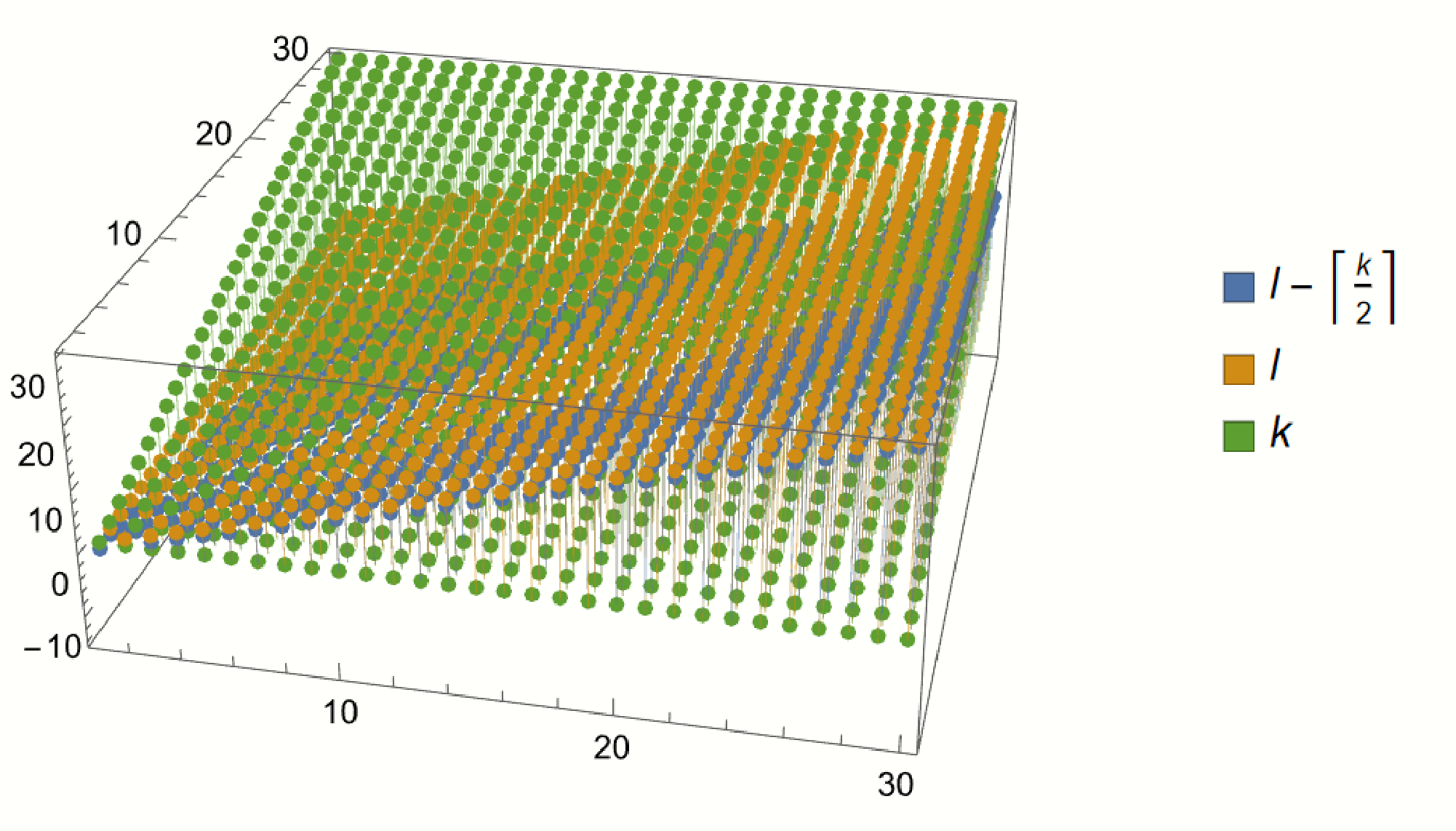}
  \caption{\mbox{The generalized (edge-)connectivity} of $S(n,\ell)$}
\label{Fig:edge-VerConn}
\end{figure}

The number of spanning tree of
$G$ denoted by $\tau(G)$. Let
\begin{equation}\label{Equ:Entropy}
\rho(G)=\lim _{V(G) \longrightarrow \infty} \frac{\ln |\tau(G)|}{\left|V(G)\right|},
\end{equation}
where $\rho(G)$ is called the entropy of spanning trees or the asymptotic complexity \cite{Burton1993,Dehmer2017}.

As an application of generalized (edge-)connectivity,
similarly to the Equation \ref{Equ:Entropy},
it can describe the fault tolerance of a graph or network,
a common metric is called   the entropy of spanning trees.

we give the entropy of the $k$-Steiner tree of a graph $G$ can be defined as
$$
\rho_k(G)=\lim _{|V(G)| \longrightarrow \infty} \frac{\ln |\kappa_{k}(G)|}{\left|V(G)\right|}
$$
The entropy of the $3,6,9$-Steiner tree of
Sierpi\'{n}ski graph $S(8,\ell)$ can be seen in Figure \ref{Fig:ROu}.
\begin{figure}[!htbp]
  \centering
  \includegraphics[width=10cm]{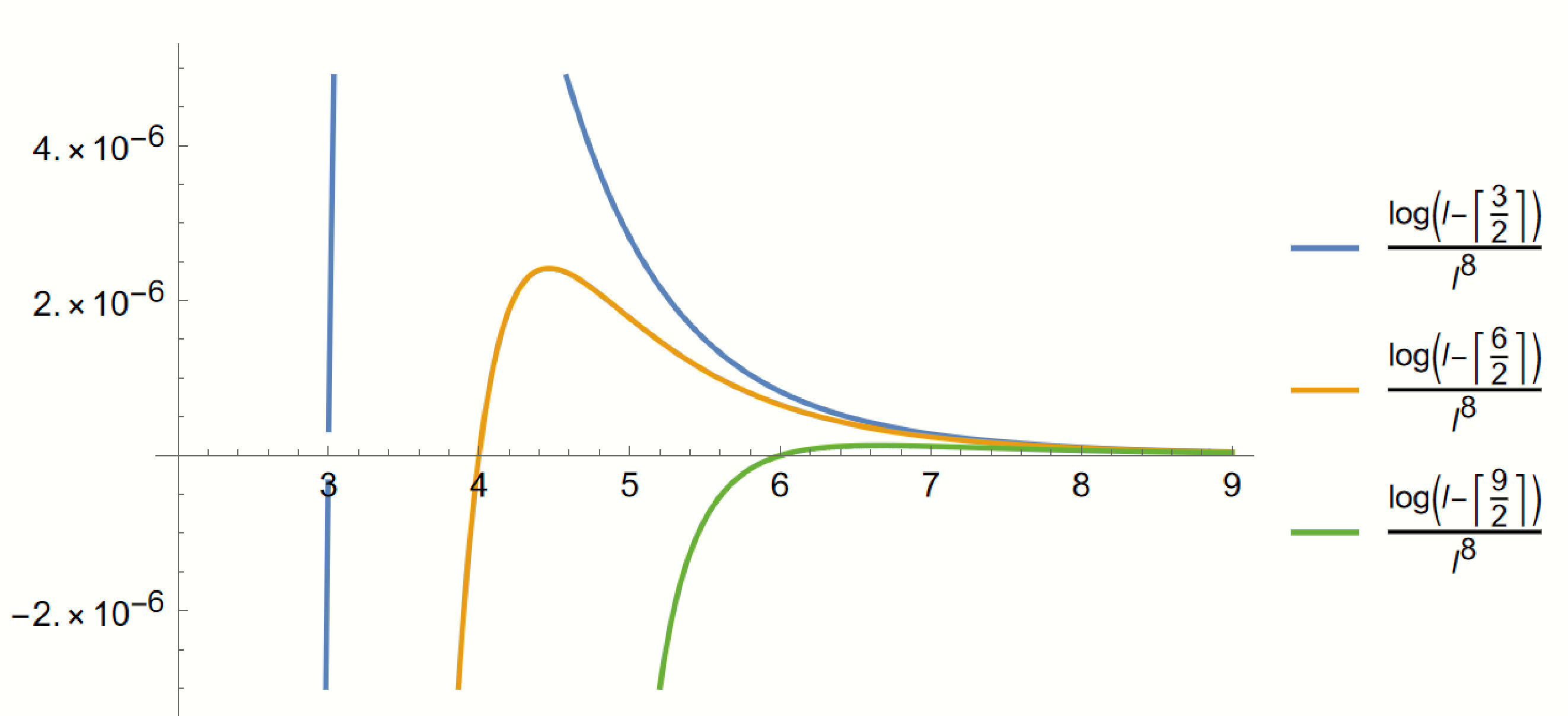}
  \caption{Entropy of $S(n,\ell)$for
   $n=8$, $k=3,6,9$ and $\ell\rightarrow +\infty$}\label{Fig:ROu}
\end{figure}

The definition of clustering coefficient can be found in \cite{Burton03}.
Let $N_v(t)$ be the number of edges in $G_t$ among neighbors of $v$,
which is the number of triangles connected to the vertex $v$.
The clustering coefficient of a graph is based on a local clustering coefficient for each vertex
$$
C[v]=\frac{N_v(t) }{ d_G(v)(d_G(v)-1)/2 },
$$
If the degree of node $v$ is $0$ or $1$, then
we can set $C[v]=0$.
By definition, we have $0 \leq C[v] \leq 1$
for $v\in V(G)$.

The clustering coefficient for the whole graph
$G$ is the average of the local values $C(v)$
$$
C(G)=\frac{1}{|V(G)|}\left( \sum_{v\in V(G)}C[v]  \right).
$$
The clustering coefficient of a graph is closely related to the transitivity of a graph,
as both measure the relative frequency of triangles\cite{Reichardt2008,Wangxiaofan2019}.

\begin{proposition}
The clustering coefficient of generalized Sierpi\'{n}ski graph $S(n,\ell)$ is
$$
C(S(n,\ell))=\frac{\ell^{- n} \left(2 \ell -
2\ell^{n} + \ell^{n + 1}\right)}{\ell}.
$$
\end{proposition}
\begin{proof}
For any $v\in V(S(n,\ell))$,
if $v$ is a extremal vertex,
then $d_{S(n,\ell)}(v)=\ell-1$ and $G[\{N(v)\}]\cong K_{\ell-1}$,
and hence
$$C[v]=\frac{N_v(t) }{ d_G(v)(d_G(v)-1)/2 }=
{\binom{\ell-1}{2}}/{\binom{\ell-1}{2}}=1.$$
If $v$ is not a extremal vertex,
then $d_{S(n,\ell)}(v)=\ell$ and $G[\{N(v)\}]\cong K_{\ell}+e$,
where $K_{\ell}+e$ is graph obtained from a
complete graph $K_{\ell}$ by adding a pendent edge.
Hence, we have
$C[v]=
{\binom{\ell-1}{2}}/{\binom{\ell}{2}}
=\frac{\ell-2}{\ell}$.

Since there exists $\ell$ extremal
vertices in Sierpi\'{n}ski graph
$S(n,\ell)$,
it follows that
$$
C(S(n,\ell))=\frac{1}{|V(G)|} \left( \sum_{v\in V(G)}C[v]  \right)=
\frac{1}{\ell^n}
\left({\ell\times 1+(\ell^n-\ell)\frac{\ell-2}{\ell}
}\right)
=\frac{\ell^{- n} \left(2 \ell -
2\ell^{n} + \ell^{n + 1}\right)}{\ell}
 $$
\end{proof}
\begin{figure}[!htbp]
  \centering
\begin{minipage}[t]{0.45\textwidth}
\center
  \hspace{30pt}
  \includegraphics[width=7cm]{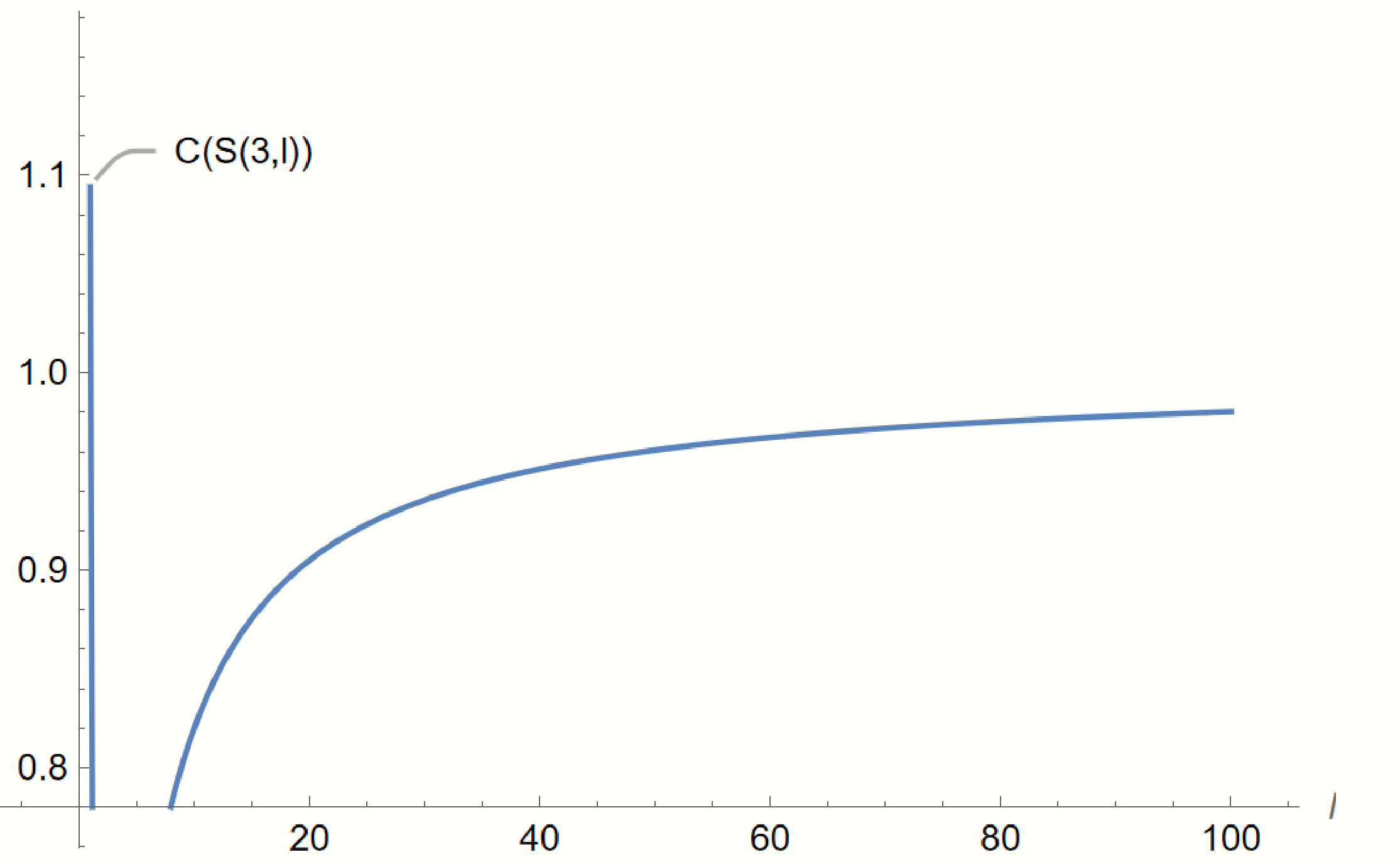}
  \caption{The Function of $C_{S(3,l)}$  }\label{Fig:CLDI}
  \end{minipage}
  \hspace{20pt}
\begin{minipage}[t]{0.43\textwidth}
\center
  \vspace{9pt}
  \includegraphics[width=7cm]{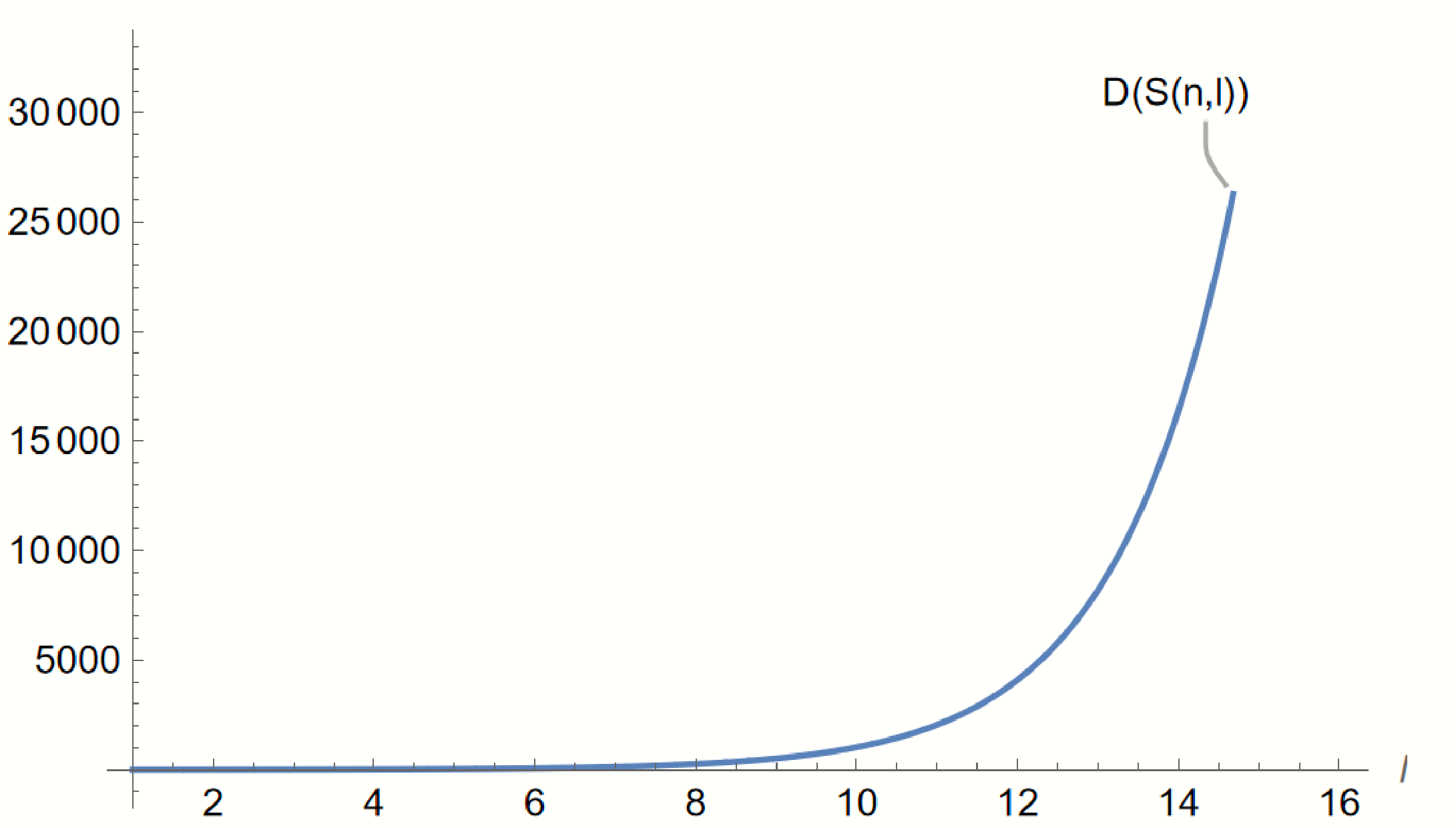}
  \caption{The diameter of $S(n,\ell)$}
  \label{Fig:DiamS}
  \end{minipage}
\end{figure}

\begin{theorem}{\upshape \cite{Parisse03}}
The diameter of $S(n,\ell)$ is
$Diam(S(n,\ell))=2^{\ell}-1$;
\end{theorem}

For network properties of Sierpi\'{n}ski graph
$S(n,\ell)$, the the diameter function
can be seen in Figure \ref{Fig:DiamS}
and its clustering coefficient is closely related to $1$ when
$\ell \longrightarrow \infty$; see Figure \ref{Fig:CLDI},
which implies that the Sierpi\'{n}ski graph
$S(n,\ell)$ is a hight transitivity graph.

\section{Acknowledgments}

The authors would like to thank the anonymous referee for the careful reading of the manuscript and the numerous helpful suggestions and comments.

\end{document}